\documentclass[12pt]{article}
\usepackage{amsmath, amssymb, latexsym}
\usepackage{amsopn,amsfonts, amsthm}
\usepackage[T1]{fontenc}

\newtheorem{theorem}{Theorem}[section]

\newtheorem{corollary}[theorem]{Corollary}

\newtheorem{proposition}[theorem]{Proposition}
\newtheorem{lemma}[theorem]{Lemma}
\theoremstyle{definition}

\newtheorem{definition}[theorem]{Definition}

\newtheorem{problem}[theorem]{Problem}
\newtheorem{question}[theorem]{Question}
\newtheorem{remark}[theorem]{Remark}

\begin{document}

\title{Denumerable cellular families in Hausdorff spaces and
towers of Boolean algebras in $\mathbf{ZF}$}
\author{Kyriakos Keremedis, Eliza Wajch \\
\\
Department of Mathematics, University of the Aegean\\
Karlovassi, Samos 83200, Greece\\
kker@aegean.gr\\
Institute of Mathematics\\
Faculty of Exact and Natural Sciences \\
Siedlce University of Natural Sciences and Humanities\\
ul. 3 Maja 54, 08-110 Siedlce, Poland\\
eliza.wajch@wp.pl\\
Orcid: 0000-0003-1864-2303}
\maketitle

\begin{abstract}
A denumerable cellular family of a topological space $\mathbf{X}$ is an
infinitely countable collection of pairwise disjoint non-empty open sets of $%
\mathbf{X}$. \textit{\ }\smallskip\ It is proved that the following
statements are equivalent in $\mathbf{ZF}$:

(i) For every infinite set $X,[X]^{<\omega }\mathbf{\ }$has a denumerable
subset.\smallskip

(ii) Every infinite $0$-dimensional Hausdorff space admits a denumerable
cellular family.\smallskip\ 

It is also proved that (i) implies the following:

(iii) Every infinite Hausdorff Baire space has a denumerable cellular family.

Among other results, the following theorems are also proved in $\mathbf{ZF}$%
:\smallskip

(iv) Every countable collection of non-empty subsets of $\mathbb{R}$ has a
choice function iff, for every infinite second-countable Hausdorff space $%
\mathbf{X}$, it holds that every base of $\mathbf{X}$ contains a denumerable
cellular family of $\mathbf{X}$.\smallskip

(v) If every Cantor cube is pseudocompact, then every non-empty countable
collection of non-empty finite sets has a choice function.\smallskip

(vi) If all Cantor cubes are countably paracompact, then (i) holds.\smallskip

Moreover, among other forms independent of $\mathbf{ZF}$, a partial
Kinna-Wagner selection principle for families expressible as countable
unions of finite families of finite sets is introduced. It is proved that if
this new selection principle and (i) hold, then every infinite Boolean
algebra has a tower and every infinite Hausdorff space has a denumerable
cellular family. \medskip

\end{abstract}

\section{Introduction}
\label{intro}

In this paper, the intended context for reasoning and statements of theorems
is the Zermelo-Fraenkel set theory $\mathbf{ZF}$ with neither the axiom of
choice $\mathbf{AC}$ nor its weaker form, unless otherwise noted. As usual, $%
\omega$ denotes the set of all finite ordinal numbers of von Neumann. The
set $\mathbb{N}=\omega\setminus\{0\}$ is the set of all natural numbers of $%
\mathbf{ZF}$. If $n\in\omega$, then $n+1=n\cup\{ n\}$. To avoid
misunderstanding, let us recall several concepts concerning infiniteness,
and introduce convenient notation.

\begin{definition}
\label{d1.1} A set $X$ is called:

\begin{enumerate}
\item[(i)] \textit{infinitely countable} or, equivalently, \textit{%
denumerable}, if $X$ is equipotent with $\omega$;

\item[(ii)] \textit{finite} if there exists $n\in\omega$ such that $n$ is
equipotent with $X$;

\item[(iii)] \textit{countable} if $X$ is equipotent with a subset of $%
\omega $;

\item[ (iv)] \textit{Dedekind-infinite} if $X$ contains a denumerable subset;

\item[(v)] \textit{Dedekind-finite} if $X$ is not Dedekind-infinite;

\item[(vi)] $n$-Dedekind-infinite for $n\in\omega\setminus\{0, 1\}$ if the
set $[X]^n$ of all $n$-element subsets of $X$ is Dedekind-infinite;

\item[(vii)] \textit{weakly Dedekind-infinite} if the power set $\mathcal{P}%
(X)$ of $X$ is Dedekind-infinite:

\item[(viii)] \textit{quasi Dedekind-infinite} if the set $[X]^{<\omega }$
of all finite subsets of $X$ is Dedekind-infinite.
\end{enumerate}
\end{definition}

The following forms are all independent of $\mathbf{ZF}$ where, in $\mathbf{%
IDI}_n$, $n$ is a fixed natural number:

\begin{itemize}
\item $\mathbf{IDI}$ (Form 9 of \cite{hr}): Every infinite set is
Dedekind-infinite.

\item $\mathbf{IWDI}$ (Form 82 of \cite{hr}): Every infinite set is weakly
Dedekind-infinite.

\item $\mathbf{IQDI}$: Every infinite set is quasi Dedekind-infinite.

\item $\mathbf{IDI}_n$: Every infinite set is $n$-Dedekind-infinite.

\item $\mathbf{IDI}_{F}$: For every infinite set $X$, there exists $n\in 
\mathbb{N}$ such that $[X]^{n}$ is Dedekind-infinite.
\end{itemize}

Let us recall the following definition:

\begin{definition}
\label{d1.3}

\begin{enumerate}
\item[(i)] A collection $\mathcal{A}$ of sets is called a \textit{disjoint
family} if, for every pair $A,B$ of distinct sets from $\mathcal{A}$, $A\cap
B=\emptyset$.

\item[(ii)] A \textit{cellular family} of a topological space $\mathbf{X}$
is a disjoint family of non-empty open sets of $\mathbf{X}$.

\item[(iii)] If $N\subseteq\omega$, then a collection $\{U_n: n\in N\}$ of
non-empty open subsets of a topological space $\mathbf{X}$ is called \textit{%
cellular} if $U_m\cap U_n=\emptyset$ for each pair $m,n$ of distinct members
of $N$.
\end{enumerate}
\end{definition}

This article is about conditions for Hausdorff spaces to admit
denumerable cellular families and about conditions for Boolean algebras to have towers. In Section 2, we establish basic notation and
mainly relatively simple preliminary results. It is known, for instance,
from \cite{et} and \cite{kt} that it is independent of $\mathbf{ZF}$ that
every denumerable compact Hausdorff space admits a denumerable cellular
family. Even the sentence that all infinite discrete spaces admit
denumerable cellular families is independent of $\mathbf{ZF}$ because it is
equivalent to $\mathbf{IWDI}$ (see \cite{kt}). In Section 2, it is shown
that $\mathbf{IWDI}$ is equivalent to the following sentence:

\begin{itemize}
\item $\mathbf{IMS}(cell,\aleph _{0})$: Every infinite metrizable space
admits a denumerable cellular family.
\end{itemize}

Furthermore, in Section 2, among other facts, we remark that every
topological space which does not admit a denumerable cellular family is
pseudocompact. It follows from $\mathbf{IDI}$ that every topological space
which is not lightly compact admits a denumerable cellular family.

The first non-trivial new result of Section 3 asserts that it holds in $%
\mathbf{ZF}$ that if a topological space $\mathbf{X}$ has a denumerable
locally finite family of open sets, then $\mathbf{X}$ admits a denumerable
locally finite cellular family. Among other facts established in Section 3,
we strengthen a result from Section 2 by showing that $\mathbf{IQDI}$
implies that every infinite topological space which is not lightly compact
admits a denumerable cellular family.

In \cite{kt}, it was proved that $\mathbf{IQDI}$ is equivalent to the
sentence: for every infinite set $X$, the Cantor cube $\mathbf{2}^{X}$ has a
denumerable cellular family. However, this result does not answer the
following question:

\begin{question}
\label{q1} Does $\mathbf{IQDI}$ imply that, for every infinite set $X$,
every subspace of the Cantor cube $\mathbf{2}^X$ admits a denumerable
cellular family?
\end{question}

In Section 4, we answer Question \ref{q1} in the affirmative by proving that 
$\mathbf{IQDI}$ is equivalent to the following sentence:

\begin{itemize}
\item $\mathbf{I0}dim\mathbf{HS}(cell, \aleph_0)$ (see \cite{kt}): Every
infinite zero-dimensional Hausdorff space admits a denumerable cellular
family.
\end{itemize}

Similarly to the authors of \cite{kt}, we also turn our attention to the
following sentence:

\begin{itemize}
\item $\mathbf{IHS}(cell, \aleph_0)$ (see \cite{kt}): Every infinite
Hausdorff space admits a denumerable cellular family.
\end{itemize}

In \cite{kt}, the following open problem was posed and left unsolved:

\begin{problem}
\label{q2} Does $\mathbf{IQDI}$ imply $\mathbf{IHS}(cell, \aleph_0)$?
\end{problem}

Although we are unable to give a satisfactory solution to Problem \ref{q2},
we prove in Section 4 that $\mathbf{IQDI}$ implies that every infinite
Hausdorff space which is also a Baire space admits a denumerable cellular
family. We also prove that $\mathbf{IQDI}$ implies that every infinite
Hausdorff space which has a well-orderable dense set admits a denumerable
cellular family. To answer Question \ref{q1} and give a deeper insight into
Problem \ref{q2}, we introduce and investigate in Section 4 useful concepts
of a regular matrix and a clopen matrix of a Hausdorff space.

In \cite{kt}, the following question was also asked:

\begin{question}
\label{q3} Does $\mathbf{IDI}$ imply the sentence ``For every infinite
Hausdorff space $\mathbf{X}$, every base of $\mathbf{X}$ contains a
denumerable cellular family of $\mathbf{X}$''?
\end{question}

In Section 4, we show a model of $\mathbf{ZF+IDI}$ in which even the Cantor
cube $\mathbf{2}^{\omega}$ has a base which does not contain any denumerable
cellular family of $\mathbf{2}^\omega$. In Section 4, we also consider the
following new sentence:

\begin{itemize}
\item $\mathbf{IQDI}(\mathcal{P})$: For every infinite set $X$, $\mathcal{P}%
(X)$ is quasi Dedekind-infinite.
\end{itemize}

We prove that $\mathbf{IQDI}(\mathcal{P})$ holds if and only if every
infinite discrete space has a clopen matrix. In consequence, $\mathbf{IQDI}(%
\mathcal{P})$ is independent of $\mathbf{ZF}$.

Section 5 is about the problem of whether Cantor cubes can fail to be
pseudocompact or countably paracompact. In Section 5, we apply some results
of Sections 2-4 to prove that if $\mathcal{M}$ is a model of $\mathbf{ZF}$
in which there exists a denumerable disjoint family of non-empty finite sets
without a partial choice function, then there exists in $\mathcal{M}$ a
metrizable Cantor cube which is not pseudocompact. It is also shown in
Section 5 that if $\mathbf{IQDI}$ fails, then there are Cantor cubes that
are not countably paracompact. However, all metrizable Cantor cubes are
paracompact in $\mathbf{ZF}$.

In Section 6, we prove that, for natural numbers $k, m$ such that $k<m$, $%
\mathbf{IDI}_k$ implies $\mathbf{IDI}_m$; furthermore, we prove that $%
\mathbf{IDI}_F$ implies $\mathbf{IQDI}$ and the axiom of countable multiple
choice implies $\mathbf{IQDI}$. We recall that the axiom of countable
multiple choice is the following sentence:

\begin{itemize}
\item $\mathbf{CMC}$ (Form 126 in \cite{hr}): For every denumerable set $X$
of non-empty sets, there exists a function $f: X\to\mathcal{P}(\bigcup X)$
such that, for every $x\in X$, $f(x)$ is a non-empty finite subset of $x$.
\end{itemize}

In Section 6, we also show a model of $\mathbf{ZFA}$ in which $\mathbf{IQDI}$
holds but $\mathbf{IDI}_F$ fails. Moreover, we show a model of $\mathbf{ZF}$
in which there exists an infinite Boolean algebra $\mathcal{B}$ which has a
tower but there is an infinite Boolean subalgebra of $\mathcal{B}$ which
fails to have a tower. We show in Section 6 that $\mathbf{IQDI}$ implies
that every infinite Boolean algebra has a tower if and only if every
infinite Boolean algebra expressible as a denumerable union of finite sets
has a tower. In Section 6, we use the following new modifications of the
familiar Kinna-Wagner selection principle for families of finite sets (see
Form [\textbf{62 E}] in \cite{hr}):

\begin{itemize}
\item $\mathbf{PKW}(\infty,<\aleph_0)$ (Kinna-Wagner partial selection
principle for families of finite sets): For every non-empty set $J$ and
every family $\{A_j: j\in J\}$ of finite sets such that $\vert A_j\vert\ge 2$
for every $j\in J$, there exist an infinite subset $I$ of $J$ and a family $%
\{B_j: j\in I\}$ of non-empty sets such that, for every $j\in I$, $B_j$ is a
proper subset of $A_j$.

\item $\mathbf{QPKW}(\infty, <\aleph_0)$: For every non-empty set $J$ and
every family $\{A_j: j\in J\}$ of finite sets such that $\vert A_j\vert\ge 2$
for every $j\in J$, if $J$ is a countable union of finite sets, then there
exist an infinite subset $I$ of $J$ and a family $\{B_j: j\in I\}$ of
non-empty sets such that, for every $j\in I$, $B_j$ is a proper subset of $%
A_j$.
\end{itemize}

One can observe that $\mathbf{PKW}(\infty,<\aleph_0)$ is a restriction to
families of finite sets of the separation principle $\mathbf{SP}^{-}$
investigated in \cite{CruzP}. The principle $\mathbf{SP}^{-}$ can be found
as Form 379 in \cite{hr} where it is denoted by $\mathbf{PKW}(\infty,
\infty, \infty)$.

We conclude that, in every model $\mathcal{M}$ of $\mathbf{ZF}+\mathbf{QPKW}%
(\infty, <\aleph_0)$, the statement $\mathbf{IQDI}$ implies that every
infinite Boolean algebra has a tower and, in consequence, $\mathbf{IQDI}$, $%
\mathbf{I0}dim\mathbf{HS}(cell, \aleph_0)$ and $\mathbf{IHS}(cell, \aleph_0)$
are all equivalent in $\mathcal{M}$. We finish by remarks on the
set-theoretic strength of the new separation principles.

For readers' convenience, we list below some of the not defined above weak
forms of the axiom of choice we shall deal with in the sequel. Several other
forms are included and discussed in Section 6.

\begin{itemize}
\item If $n\in \omega \backslash \{0,1\}$, $C(\omega,n)$ (Form 288(n) of 
\cite{hr}): Every denumerable family of $n$-element sets has a choice
function.

\item $(\forall n\in \omega\setminus\{0, 1\})C(\omega,n)$: For each $n\in
\omega \setminus\{0, 1\}$, every denumerable family of $n$-element sets has
a choice function.

\item $\mathbf{CAC}$ (Form 8 in \cite{hr}): Every denumerable family of
non-empty sets has a choice function.

\item $\mathbf{CAC}_{fin}$ (Form 10 of \cite{hr}): $\mathbf{CAC}$ restricted
to families of finite sets. Equivalently, every denumerable family of
non-empty finite sets has an infinite subfamily with a choice function.

\item $\mathbf{CAC}(\mathbb{R})$ (Form 94 of \cite{hr}): Every denumerable
family of non-empty subsets of $\mathbb{R}$ has a choice function.
Equivalently, every denumerable family of non-empty subsets of $\mathbb{R}$
has an infinite subfamily with a choice function (see \cite{hkrst}).

\item $\mathbf{CAC}_D(\mathbb{R})$: Every disjoint denumerable family of
dense subsets of $\mathbb{R}$ has a choice function (see Theorem 3.14 of 
\cite{kw}).

\item $\mathbf{MC}$ (Form 67 in \cite{hr}): For every disjoint family $%
\mathcal{A}=\{A_{i}: i\in I\}$ of non-empty sets there exists a family of
non-empty finite sets $\mathcal{B}=\{B_{i}: i\in I\}$ such that, for each $%
i\in I$, $B_{i}\subseteq A_{i}$.

\item $E(I,Ia)$ (Form 64 in \cite{hr}): There are no amorphous sets.

\item $\mathbf{MP}$ (Form 383 in \cite{hr}): Every metrizable space is
paracompact.

\item $\mathbf{PKW}(\infty, \infty, \infty)$ (Form 379 in \cite{hr}, $%
\mathbf{SP^{-}}$ in \cite{CruzP}): For every infinite family $\mathcal{F}$
of non-empty sets with at least two elements each, there exist an infinite
subfamily $\mathcal{F}^{^{\prime }}$ of $\mathcal{F}$ and a function
assigning a non-empty proper subset to each element of $\mathcal{F}%
^{^{\prime }}$.

\item $\mathbf{KW}(\aleph_0, <\aleph_0)$ (Form 358 of \cite{hr}): For every
denumerable set $X$ of finite sets, there exists a function $f:X\to\bigcup\{%
\mathcal{P}(A): A\in X\}$ such that, for every $A\in X$, if $\vert A\vert>1$%
, then $f(A)$ is a non-empty proper subset of $A$.
\end{itemize}

To stress the fact that a result is proved in $\mathbf{ZF}$ we shall write
at the beginning of the statements of the theorems and propositions ($%
\mathbf{ZF}$). Apart from models of $\mathbf{ZF}$, we refer to some models
of $\mathbf{ZFA}$, i.e., $\mathbf{ZF}$ with atoms (see \cite{J1} and \cite%
{J2}). The system $\mathbf{ZFA}$ is denoted by $\text{ZF}^{0}$ in \cite{hr}.
All our theorems of $\mathbf{ZF}$ are also theorems of $\mathbf{ZFA}$.

\section{Preliminaries}
\label{s2}

\subsection{Notation and terminology}

In the sequel, boldface letters will denote topological spaces and lightface
letters will denote their underlying sets, that is, a topological space $(X, 
\mathcal{T})$ will be denoted by $\mathbf{X}$. For a subset $A$ of a
topological space $\mathbf{X}$, we denote by $int(A)$ the interior of $A$,
by $\overline A$ the closure of $A$, and by $\partial(A)$ the boundary of $A$
in $\mathbf{X}$. That a set $A$ is a proper subset of a set $B$ is denoted
by $A\subset B$.

Let us recall several definitions.

\begin{definition}
\label{d2.1} Let $\mathcal{U}$ be a collection of subsets of a topological
space $\mathbf{X}$. Then $\mathcal{U}$ is called:

\begin{enumerate}
\item[(i)] \textit{point-finite} if, for every point $x\in X$, the set $%
\{U\in\mathcal{U}: x\in U\}$ is finite;

\item[(ii)] \textit{locally finite} if every point of $\mathbf{X}$ has a
neighborhood which meets only finitely many members of $\mathcal{U}$.
\end{enumerate}
\end{definition}

\begin{definition}
\label{d2.2} A topological space $\mathbf{X}$ is called:

\begin{enumerate}
\item[(i)] \textit{compact} (resp. \textit{countably compact}) if every open
cover (resp., countable open cover) of $\mathbf{X}$ has a finite subcover;

\item[(ii)] \textit{lightly compact} if every locally finite family of open
subsets of $\mathbf{X}$ is finite;

\item[(iii)] \textit{pseudocompact} if every continuous function from $%
\mathbf{X}$ to $\mathbb{R}$ is bounded;

\item[(iv)] \textit{dense-in-itself} if $\mathbf{X}$ does not have isolated
points;

\item[(v)] \textit{zero-dimensional} or, equivalently, $0$-\textit{%
dimensional} if $\mathbf{X}$ has a base consisting of clopen (simultaneously
closed and open) subsets of $\mathbf{X}$.
\end{enumerate}
\end{definition}

\begin{definition}
\label{d2.3} Let $\mathcal{F}$ be a collection of non-empty subsets of a
topological space $\mathbf{X}$. Then:

\begin{enumerate}
\item[(i)] $\mathcal{F}$ is called a \textit{filter base} on $\mathbf{X}$
if, for every pair $A,B$ of members of $\mathcal{F}$, there exists $C\in%
\mathcal{F}$ such that $C\subseteq A\cap B$;

\item[(ii)] if $\mathcal{F}$ is a filter base, then the \textit{adherence}
of $\mathcal{F}$ is the set $\bigcap\{\overline{A}: A\in\mathcal{F}\}$;

\item[(iii)] if $\mathcal{F}$ is infinite, then an element $x\in X$ is
called a \textit{cluster point} of $\mathcal{F}$ if, for every neighborhood $%
U$ of $x$ in $\mathbf{X}$, the set $\{A\in\mathcal{F}: A\cap
U\neq\emptyset\} $ is infinite.
\end{enumerate}
\end{definition}

\begin{definition}
\label{d2.4} A subset $A$ of a topological space $\mathbf{X}$ is called 
\textit{regular open} in $\mathbf{X}$ if $A=int(\overline{A})$.
\end{definition}

\begin{definition}
\label{d2.5} For a topological space $\mathbf{X}$, we denote by:

\begin{enumerate}
\item[(i)] $RO(\mathbf{X})$ the collection of all regular open sets of $%
\mathbf{X}$, as well as the Boolean algebra $(RO(\mathbf{X}), \vee, \wedge,
^{\prime }, \mathbf{0}, \mathbf{1})$ of all regular open sets of $\mathbf{X}$
where:

\begin{itemize}
\item $\mathbf{0}=\emptyset$ and $\mathbf{1}=X,$
\end{itemize}

and if $U, V\in RO(\mathbf{X})$, then:

\begin{itemize}
\item $U\wedge V=U\cap V,$

\item $U\vee V=int(\overline{U\cup V}),$

\item $U^{\prime }=X\backslash \overline{U}$;
\end{itemize}

\item[(ii)] $Clop(\mathbf{X})$ the collection of all clopen subsets of $%
\mathbf{X}$, as well as the Boolean subalgebra $(Clop(\mathbf{X}), \vee,
\wedge, ^{\prime }, \mathbf{0}, \mathbf{1})$ of the Boolean algebra of all
regular open sets of $\mathbf{X}$.
\end{enumerate}
\end{definition}

\begin{remark}
\label{r2.6} Let $\mathbf{X}=(X,\mathcal{T})$ be a topological space.

\begin{enumerate}
\item[(i)] Clearly, if $U,V\in Clop(\mathbf{X})$, then $U\vee V=U\cup V$, so
the Boolean algebra $Clop(\mathbf{X})$ is a Boolean subalgebra of the power
set Boolean algebra $\mathcal{P}(X)$.

\item[(ii)] The collection $RO(\mathbf{X})$ is a base of a topology $%
\mathcal{R}$ on $X$ such that $\mathcal{R}\subseteq \mathcal{T}$. The
topology $\mathcal{R}$ is called the \textit{semi-regularization} of $%
\mathcal{T}$. In case where $\mathbf{X}$ is Hausdorff and $\mathcal{T}=%
\mathcal{R}$, the space $\mathbf{X}$ is called \textit{semi-regular}.
Evidently, every regular Hausdorff space is semi-regular, but there exist
semi-regular non regular spaces (see, e.g., \cite{ss}, Example 81).

\item[(iii)] It is well known that the Boolean algebra $RO(\mathbf{X})$ is
complete.
\end{enumerate}
\end{remark}

\begin{definition}
\label{d2.7} Let $\mathcal{P}=(P,\leq )$ be a poset (a partially ordered
set). Then:

\begin{enumerate}
\item[(i)] a strictly $\leq $-decreasing sequence $(t_{n})_{n\in \omega }$
in $P$ is called a \textit{tower} of $\mathcal{P}$;

\item[(ii)] a family $\mathcal{C}$ of elements of $P$ is called an \textit{%
antichain} of $\mathcal{P}$ if, for all $c,d\in \mathcal{C}$ with $c\neq d$, 
$c\ $and $d$ are not compatible, i.e., there does not exist $p\in P$ such
that $p\leq c$ and $p\leq d$.
\end{enumerate}
\end{definition}

\begin{definition}
\label{d2.8} Let $\mathcal{B}=(B,+,\cdot ,^{\prime },\mathbf{0},\mathbf{1})$
be a Boolean algebra.

\begin{enumerate}
\item[(i)] The binary relation $\leq $ on $B$ given by 
\begin{equation}
x\leq y\leftrightarrow x=x\cdot y  \label{0}
\end{equation}
is called the \textit{partial order} of $\mathcal{B}$.

\item[(ii)] A family $\mathcal{C}$ of non-zero elements of $\mathcal{B}$ is
called an \textit{antichain} of $\mathcal{B}$ if $\mathcal{C}$ is an
antichain of the poset $(B,\leq )$.

\item[(iii)] Every tower of $(B,\leq )$ is called a \textit{tower of the
Boolean algebra} $\mathcal{B}$.
\end{enumerate}
\end{definition}

\begin{definition}
\label{d2.9} Let $X$ be a non-empty set.

\begin{enumerate}
\item[(i)] We denote by $\mathbf{2}$ the discrete space $(2, \mathcal{P}(2))$
where $2=\{0, 1\}$.

\item[(ii)] $\mathbf{2}^{X}$ denotes the Tychonoff product of the discrete
space $\mathbf{2}$, i.e., $\mathbf{2}^X$ is a \textit{Cantor cube}.

\item[(iii)] $Fn(X,2)$ is the set of all finite partial functions from $X$
into $2$, i.e., $p\in Fn(X, 2)$ iff there exists a non-empty set $A\in
[X]^{<\omega}$ such that $p$ is a function from $A$ into 2.

\item[(iv)] For $p\in Fn(X, 2)$, $\lbrack p]=\{f\in 2^{X}: p\subseteq f\}$.

\item[(v)] The collection $\mathcal{B}(X)=\{[p]: p\in Fn(X,2)\}$ is called
the \textit{standard base} of $\mathbf{2}^X$.
\end{enumerate}
\end{definition}

\subsection{Preliminary results}

The following proposition is well-known (see, e.g., \cite{fz}, \cite{FrZb}, 
\cite{hod}, \cite{kb}).

\begin{proposition}
\label{p2.10} $\mathbf{(ZF)}$ The following hold:

\begin{enumerate}
\item[(a)] A Boolean algebra has a denumerable antichain iff it has a tower.

\item[(b)] Let $\mathbf{X}=(X,\mathcal{T})$ be a topological space. Then the
following hold:

\begin{enumerate}
\item[(i)] If $\mathbf{X}$ is Hausdorff and $\mathcal{R}$ is the
semi-regularization of $\mathcal{T}$, then $(X,\mathcal{R})$ is Hausdorff.

\item[(ii)] The Boolean algebra $RO(\mathbf{X})$ has a tower iff it has a
denumerable antichain.

\item[(iii)] $\mathbf{X}$ has a denumerable cellular family iff $RO(\mathbf{X%
})$ has a tower.

\item[(iv)] For all $U,V\in RO(\mathbf{X})$, if $V\subset U$, then $%
U\backslash \overline{V}\neq \emptyset $.

\item[(v)] The Boolean algebra $Clop(\mathbf{X})$ has a tower iff it has a
denumerable antichain.
\end{enumerate}

\item[(c)] If, for every infinite Hausdorff space $\mathbf{X}$, there exists
a tower of the Boolean algebra $RO(\mathbf{X})$, then $\mathbf{IHS}%
(cell,\aleph _{0})$ holds.
\end{enumerate}
\end{proposition}

\begin{corollary}
\label{c02.11} For every topological space $\mathbf{X}$, the following
conditions are satisfied:

\begin{enumerate}
\item[(i)] $RO(\mathbf{X})$ has a tower iff $\mathbf{X}$ has a denumerable
cellular family of regular open sets;

\item[(ii)] $Clop(\mathbf{X})$ has a tower iff $\mathbf{X}$ has a
denumerable cellular family of clopen sets.
\end{enumerate}
\end{corollary}

An easy proof to part (ii) of the following Proposition \ref{p02.12} is left
to readers as an exercise.

\begin{proposition}
\label{p02.12} $\mathbf{(ZF)}$

\begin{enumerate}
\item[(i)] \cite{kt} An infinite Boolean algebra is Dedekind-infinite iff it
has a denumerable antichain (iff it has a tower, by Proposition \ref{p2.10}%
(a)). In particular, for every infinite Hausdorff space $\mathbf{X}$, if $RO(%
\mathbf{X})$ has a denumerable subset then $\mathbf{X}$ has a denumerable
cellular family.

\item[(ii)] Assume $\mathbf{IQDI}$. A Boolean algebra has a denumerable
antichain (resp. denumerable chain) iff it has an infinite antichain (resp.
infinite chain). In particular, for every infinite Hausdorff space $\mathbf{X%
}$, $RO(\mathbf{X})$ has a denumerable cellular family iff $\mathbf{X}$ has
an infinite cellular family.
\end{enumerate}
\end{proposition}

The following result indicates that, in $\mathbf{ZF}$, one cannot prove that
every infinite metrizable space has a denumerable cellular family:

\begin{proposition}
\label{p02.13}

\begin{enumerate}
\item[(a)] $(\mathbf{ZF})$ Every first-countable Hausdorff space which is
not discrete admits a denumerable cellular family of regular open sets.
Hence, if $\mathbf{X}$ is a non-discrete first-countable Hausdorff space,
then $RO(\mathbf{X})$ has a tower. In particular, every (quasi)-metrizable,
non-discrete space admits a denumerable cellular family.

\item[(b)] The following are equivalent in $\mathbf{ZF}$:

\begin{enumerate}
\item[(i)] $\mathbf{IWDI}$;

\item[(ii)] every infinite first-countable Hausdorff space admits a
denumerable cellular family;

\item[(iii)] $\mathbf{IMS}(cell,\aleph _{0})$;

\item[(iv)] every infinite discrete space admits a denumerable cellular
family;

\item[(v)] for every infinite set $X$, there exists a metric $d$ on $X$ such
that $(X,d)$ is not discrete.
\end{enumerate}

\item[(c)] $\mathbf{IMS}(cell,\aleph _{0})$ is not a theorem of $\mathbf{ZF}$
and it does not imply $\mathbf{IQDI}$ in $\mathbf{ZF}$.

\item[(d)] $\mathbf{(ZF)}$ Let $\mathbf{X}$ be an infinite Hausdorff space.
If $\mathbf{X}$ has a well-orderable base of clopen sets, then $\mathbf{X}$
admits a denumerable cellular family of clopen sets, so $Clop(\mathbf{X}) $
has a tower.
\end{enumerate}
\end{proposition}

\begin{proof}
($a$) Let $\mathbf{X}$ be a non-discrete, first-countable Hausdorff space. Fix an accumulation point $x_{0}$ of $\mathbf{X}$. Let $\mathcal{B}%
(x_{0})=\{U_{n}: n\in \omega \}$ be a countable base of open neighborhoods of 
$x_{0}$ in $\mathbf{X}$. We claim that 
\begin{equation*}
\mathcal{B}^{\prime }(x_{0})=\{int(\overline{U_{n}}): n\in \omega \}
\end{equation*}%
is a neighborhood base of $x_{0}$ in the semi-regularization $(X,\mathcal{R})$ of $%
\mathbf{X}$. Indeed, if $U$ is a regular open neighborhood of $x_{0}$, then
for some $n\in \omega$, $U_{n}\subseteq U$. Hence, $x_{0}\in int(%
\overline{U_{n}})\subseteq int(\overline{U})=U$. Without loss of generality,
we may assume that $int(\overline{U_{n+1}})\subset int(\overline{U_n})$ for each $n\in\omega$.
Clearly, 
\begin{equation*}
\mathcal{C}=\{int(\overline{U_{n}})\backslash \overline{int(\overline{U_{n+1}})}: n\in 
\omega\}
\end{equation*}%
is a denumerable cellular family of regular open sets of $\mathbf{X}$%
. Hence $RO(\mathbf{X})$ has a tower by Corollary \ref{c02.11}\smallskip

($b$) (i) $\rightarrow $ (ii) Fix an infinite Hausdorff space $\mathbf{X}$. If $\mathbf{X}$ has an accumulation point, then, by part (a), $\mathbf{X}$
admits a denumerable cellular family. Otherwise, by $\mathbf{IWDI}$, $X$ has
a denumerable partition $\mathcal{P}$. Clearly, the members of $\mathcal{P}$
are non-empty open sets of $\mathbf{X}$.\smallskip

(ii) $\rightarrow $ (iii), (iii) $\rightarrow $ (iv) and (iv) $\rightarrow $
(i) are straightforward.\smallskip

(i) $\leftrightarrow $ (v) This has been established in \cite{kk}.\smallskip

To prove ($c$), we notice that $\mathbf{IWDI}$ fails in model $\mathcal{M}$37 of \cite{hr}, so it follows from ($b$) that $\mathbf{IMS}(cell, \aleph_0)$ is false in $\mathcal{M}$37. Moreover, in Cohen's original model $\mathcal{M}$1 of \cite{hr},  $\mathbf{IWDI}$ holds and $\mathbf{IQDI}$ fails. To see that $\mathbf{IQDI}$ is false in $\mathcal{M}$1, let us consider the infinite set $A$ of all added Cohen reals of $\mathcal{M}$1. Then $A$ is Dedekind-finite in $\mathcal{M}$1. Hence, since  $\mathbf{CAC}_{fin}$ is true in $\mathcal{M}$1, $A$ is not quasi Dedekind-infinite in  $\mathcal{M}$1.

We omit a simple proof to ($d$) because it is similar to that of ($a$).
\end{proof}

\begin{remark}
\label{r02.14} In view of the proof to Proposition \ref{p02.13}($a$), the
following hold in $\mathbf{ZF}$:

\begin{enumerate}
\item[(i)] The semi-regularization of a first-countable space is
first-countable.

\item[(ii)] If $\mathbf{X}$ is a Hausdorff space which has an accumulation
point $x$ such that there exists a well-orderable base of neighborhoods of $%
x $, then $\mathbf{X}$ admits a denumerable cellular family of regular open
sets.
\end{enumerate}
\end{remark}

The following theorem has important consequences:

\begin{theorem}
\label{t02.15} $\mathbf{(ZF)}$ Let $S$ be a non-empty set and let $\{\mathbf{%
X}_{s}: s\in S\}$ be a collection of topological spaces such that $%
\prod_{s\in S}X_{s}\neq \emptyset $. Then $\prod_{s\in S}\mathbf{X}_{s}$
admits a denumerable cellular family if and only if there exists a non-empty
subset $T $ of $S$ such that $\prod_{s\in T}\mathbf{X}_{s}$ admits a
denumerable cellular family.\newline
In particular, if $\mathbf{X}_{t}$ admits a denumerable cellular family for
some $t\in S,$ then $\prod_{s\in S}\mathbf{X}_{s}$ admits a denumerable
cellular family.
\end{theorem}

\begin{proof}
Assume that $T$ is a non-empty subset of $S$ such that $\prod_{s\in T}%
\mathbf{X}_{s}$ admits a denumerable cellular family. Let $\{U_{n}: n\in
\omega \}$ be a cellular family of $\prod_{s\in T}\mathbf{X}_{s}$%
. For each $n\in \omega $, we define 
\begin{equation*}
V_{n}=\{x\in \prod_{s\in S}X_{s}: x|_{T}\in U_{n}\}.
\end{equation*}%
Of course, the sets $V_{n}$ are all open in $\prod_{s\in S}\mathbf{X}_{s}$
and $V_{n}\cap V_{m}=\emptyset $ for each pair $m,n$ of distinct elements of 
$\omega $. To show that all $V_{n}$ are non-empty, choose $f\in \prod_{s\in
S}X_{s}$. Fix $n\in \omega $. There exists $g\in U_{n}$. We define $%
y_{f,g}\in \prod_{s\in S}X_{s}$ as follows, if $s\in S\setminus T$, we put $%
y_{f,g}(s)=f(s)$; if $s\in T$, we put $y_{f,g}(s)=g(s)$. Then $y_{f,g}\in
V_{n}$.\smallskip
\end{proof}

\begin{corollary}
\label{c02.16} $\mathbf{(ZF)}$ For every infinite set $X$ it holds that $%
\mathbf{2}^{X}$ admits a denumerable cellular family iff for some (infinite)
subset $Y$ of $X$, $\mathbf{2}^{Y} $ admits a denumerable cellular family.
\end{corollary}

By applying Corollary \ref{c02.16}, one can easily verify, as in \cite{kt},
that:

\begin{description}
\item[(*) ] $\mathbf{IQDI}$ implies the following: For every infinite set $%
X,2^{X}$ admits a denumerable cellular family.
\end{description}

One can generalize (*) as follows:

\begin{proposition}
\label{p02.17} $\mathbf{(ZF)}$ Let $S$ be a quasi Dedekind-infinite set.
Suppose that $\{\mathbf{X}_{s}: s\in S\}$ is a collection of first-countable
Hausdorff spaces such that each $X_{s}$ consists of at least two points and $%
\prod_{s\in S}X_{s}\neq \emptyset $. Then $\mathbf{X}=\prod_{s\in S}\mathbf{X%
}_{s}$ admits a denumerable cellular family.
\end{proposition}

\begin{proof}
If there exists $s_{0}\in S$ such that $\mathbf{X}_{s_{0}}$ is not discrete,
then, by Proposition \ref{p02.13} (a), $\mathbf{X}_{s_{0}}$ admits a denumerable cellular family, so by
Theorem \ref{t02.15}, $\mathbf{X}$ admits a denumerable cellular family.

Now, suppose that, for each $s\in S$, $\mathbf{X}_{s}$ is discrete. Since $S$ is quasi Dedekind-infinite, there exists a collection $\{T_{n}: n\in \omega \}$ of
pairwise distinct finite subsets of $S$. Let $T=\bigcup_{n\in \omega }T_{n}$%
. The set $T$ is infinite, so $Y=\prod_{s\in T}X_{s}$ is also infinite. By Theorem 2.1 of \cite{ew}, $\mathbf{Y}$ is metrizable. Since $\mathbf{Y}$
has an accumulation point, $\mathbf{Y}$ admits a denumerable cellular family
by Proposition \ref{p02.13}. It follows from Theorem \ref{t02.15} that $\mathbf{X}$
admits a denumerable cellular family.\medskip
\end{proof}

The following theorem summarizes some results from \cite{kt} we shall be
needing in the present paper. The first one in the list shows that the
converse of (*) holds.

\begin{theorem}
\label{t02.18} \cite{kt} The following conditions are satisfied in $\mathbf{%
ZF}$:

\begin{enumerate}
\item[(i)] For every infinite set $X$, it holds that $X$ is quasi Dedekind
infinite iff the standard base $\mathcal{B}(X)$ of the Cantor cube $\mathbf{2%
}^{X}$ admits a denumerable cellular family iff $\mathbf{2}^{X}$ admits a
denumerable cellular family.\newline
In particular, $\mathbf{IQDI}$ iff, for every infinite set $X$, the standard
base $\mathcal{B}(X)$ of $\mathbf{2}^{X}$ contains a denumerable disjoint
subfamily iff, for every infinite set $X$, $\mathbf{2}^{X}$ admits a
denumerable cellular family.

\item[(ii)] $\mathbf{IHS}(cell,\aleph _{0})+\mathbf{CAC}_{fin}$ is
equivalent to $\mathbf{IDI}$.

\item[(iii)] $\mathbf{CMC}$ implies $\mathbf{IHS}(cell,\aleph _{0})$.

\item[(iv)] $\mathbf{IHS}(cell,\aleph _{0})$ implies $\mathbf{IQDI}$.

\item[(v)] ``Every infinite Boolean algebra has a tower'' implies $\mathbf{%
IQDI}$.
\end{enumerate}
\end{theorem}

We list the following results here for future reference.

\begin{proposition}
\label{p02.19} \cite{ker} For every topological space $\mathbf{X}$, the
following conditions are all equivalent in $\mathbf{ZF}$:

\begin{enumerate}
\item[$(B_{1})$] every countable open covering $\mathcal{U}$ of $\mathbf{X}$
has a finite subcollection $\mathcal{V}$ such that $X=\bigcup\{\overline{U}: U\in\mathcal{V}\}$;

\item[$(B_{2})$] every denumerable family $\mathcal{U}$ of non-empty open
subsets of $\mathbf{X}$ has a cluster point in $\mathbf{X}$;

\item[$(B_{3})$] every denumerable cellular family of $\mathbf{X}$ has a
cluster point in $\mathbf{X}$;

\item[$(B_{4})$] every countable filter base consisting of open sets of $%
\mathbf{X}$ has a point of adherence;

\item[$(B_{5})$] every countable, locally finite, disjoint collection of
open sets of $\mathbf{X}$ is finite.
\end{enumerate}
\end{proposition}

\begin{theorem}
\label{t02.20} \cite{ker} $(\mathbf{ZF})$ For every topological space $%
\mathbf{X}$, each of the following conditions is equivalent to\ $\mathbf{X}$
is lightly compact:

\begin{enumerate}
\item[$(A_{1})$] Every disjoint locally finite family of open sets of $%
\mathbf{X}$ is finite.

\item[$(A_{2})$] Every locally finite open cover of $\mathbf{X}$ is finite.
\end{enumerate}

In particular, every compact topological space is lightly compact (but there
are lightly compact non-compact spaces) and every paracompact, lightly
compact space is compact.
\end{theorem}

\begin{corollary}
\label{c02.21} It holds in $\mathbf{ZF}$ that $\mathbf{IDI}$ implies that
every infinite topological space $\mathbf{X}$ which does not admit a
denumerable cellular family is lightly compact.
\end{corollary}

\begin{proof} Suppose that $\mathbf{X}$ is a topological space which is not lightly compact. By Theorem \ref{t02.20}, $\mathbf{X}$ has an infinite locally finite cellular family $\mathcal{C}$. If $\mathbf{IDI}$ holds, there exists a denumerable subfamily of $\mathcal{C}$.
\end{proof}

In Section 3, we show that $\mathbf{IDI}$ can be replaced with $\mathbf{IQDI}
$ in Corollary \ref{c02.21}.

\begin{proposition}
\label{p02.22}$(\mathbf{ZF})$ Suppose that $\mathbf{X}$ is a topological
space which is not pseudocompact. Then $\mathbf{X}$ does admit a locally
finite denumerable cellular family, so $\mathbf{X}$ satisfies none of
conditions $(B_{1})-$ $(B_{5})$ of Proposition \ref{p02.19}.
\end{proposition}

\begin{proof} There exists a continuous, unbounded
real-valued function $f$ on $\mathbf{X}$. By replacing $f$ with $|f|$, we may
assume that $f(X)\subseteq [0, +\infty)$. Via a straightforward induction, we can define a
strictly increasing sequence $(k_{n})_{n\in \mathbb{N}}$ of natural numbers
such that, for every $n\in \mathbb{N}$, the set $C_n=\{x\in X: k_{n}<f(x)<k_{n+1}\}$ is non-empty.  Then $\mathcal{C}=\{C_{n}: n\in \mathbb{N}\}$ is a locally finite denumerable cellular family of $\mathbf{X}$. It follows from Proposition \ref{p02.19} that $\mathbf{X}$ satisfies none of conditions $(B_{1})-(B_{5})$ of Proposition \ref{p02.19}.
\end{proof}

\begin{proposition}
\label{p02.23}$(\mathbf{ZF})$ Suppose that $\mathbf{X}$ is an infinite
topological space which does not admit a denumerable cellular family. Then $%
\mathbf{X}$ satisfies conditions $(B_1)-(B_5)$ of Proposition \ref{p02.19}.
\end{proposition}

\begin{proof}
It suffices to check that $\mathbf{X}$ satisfies condition $(B_1)$ of Proposition \ref{p02.19}. Suppose that $\{U_{n}: n\in \omega \}$ is an open cover of $\mathbf{X}$. Let $%
V_{n}=\bigcup_{i\in n+1}U_{i}$ and $G_{n}=int(\overline{%
V_{n}})$ for each $n\in \omega $. Then $\mathcal{G}=\{G_{n}: n\in \omega \}$
is an open cover of $\mathbf{X}$ such that $G_{n}\subseteq G_{n+1}$ and $%
G_{n}\in RO(\mathbf{X})$ for each $n\in \omega $. If $\mathcal{G}$ has a
finite subcover, there exists $n_{0}\in \omega $ such that $X=G_{n_{0}}$, so 
$V_{n_{0}}$ is dense in $\mathbf{X}$.

Suppose that $\mathcal{G}$ does not have a finite subcover. Then there
exists a strictly increasing sequence $(k_{n})_{n\in \omega }$ of members of 
$\omega $ such that $G_{k_{n}}\neq G_{k_{n+1}}$ for each $n\in \omega $. Let 
$H_{n}=G_{k_{n}}$ and $A_{n}=H_{n+1}\setminus \overline{H_{n}}$ for each $%
n\in \omega $. Then $\{A_{n}: n\in \omega \}$ is a denumerable cellular
family contradicting our hypothesis. The contradiction obtained shows that $%
\mathcal{G}$ has a finite subcover and this, together with Proposition \ref{p02.19}, completes the proof.
\end{proof}

\begin{corollary}
\label{c02.24}$(\mathbf{ZF})$ If $\mathbf{X}$ is an infinite discrete space
which does not have a denumerable cellular family, then $\mathbf{X}$ is
countably compact.
\end{corollary}

\begin{proposition}
\label{p02.25} $\mathbf{(ZF)}$

\begin{enumerate}
\item[(i)] An infinite topological space admits a denumerable cellular
family iff it has an open non-pseudocompact subspace.

\item[(ii)] An infinite discrete space $\mathbf{X}$ admits a denumerable
cellular family if and only if $\mathbf{X}$ is not pseudocompact.
\end{enumerate}
\end{proposition}

\begin{proof}
(i) Fix an infinite topological space $\mathbf{X}$.

($\rightarrow $)  Let $\mathcal{U}=\{U_{n}: n\in \mathbb{N}\}$ be a
denumerable cellular family of $\mathbf{X}$. Since the function $f:\mathbf{Y}%
\rightarrow \mathbb{R}$, where $Y=\bigcup \mathcal{U},$ given by $f(x)=n$
iff $x\in U_{n},$ is continuous, it follows that the open subspace $\mathbf{Y%
}$ of $\mathbf{X}$ is not pseudocompact.

($\leftarrow $) Let $\mathbf{Y}$ be an open non-pseudocompact subspace of $%
\mathbf{X}$. It follows from Proposition \ref{p02.22} that $%
\mathbf{Y}$ has a denumerable cellular family. Hence $\mathbf{X}$ also has a denumerable cellular family. This completes the proof to (i).  

That (ii) holds follows from (i) and the fact that a discrete space is non-pseudocompact iff it has a non-pseudocompact subspace.
\end{proof}

\section{Denumerable locally finite cellular families, finite products and
denumerable point-finite families of open sets}
\label{s3}

Let us begin with the following non-trivial new theorem:

\begin{theorem}
\label{t3.1}$(\mathbf{ZF})$\ Let\ $\mathbf{X}$ be an infinite topological
space.

\begin{enumerate}
\item[(i)] $\mathbf{X}$ admits a denumerable locally finite family of clopen
sets iff $\mathbf{X}$ admits a denumerable cellular locally finite family of
clopen sets.

\item[(ii)] $\mathbf{X}$ admits a denumerable locally finite cellular family
with a dense union iff\ $\mathbf{X}$ admits a denumerable locally finite
family of open sets.

\item[(iii)] $\mathbf{X}$ admits an infinite locally finite cellular family
with a dense union iff $\mathbf{X}$ admits an infinite locally finite family
of open sets.
\end{enumerate}
\end{theorem}

\begin{proof}
(i) ($\leftarrow $) This is straightforward.\smallskip

($\rightarrow $) Suppose that $\mathcal{U}$ is a denumerable locally finite family of clopen sets of $\mathbf{X}$. By adjoining $X$ to $\mathcal{U%
}$, we may assume that $\mathcal{U}$ is a clopen cover of $\mathbf{X}$. Since $|[\omega ]^{<\omega }|=\aleph _{0}$, we
may also assume that $\mathcal{U}$ is closed under finite intersections.
Define an equivalence relation $\sim $ on $X$ by requiring: $x\sim y$ iff,
for every $U\in \mathcal{U}$, $x\in U$ iff $y\in U$. For every $x\in X$, let 
$[x]$ denote the $\sim $ equivalence class of $x$ and let $\mathcal{U}%
(x)=\{U\in \mathcal{U}: x\in U\}$. We fix $x\in X$ and claim that $[x]$ is
open. To see this, fix $y\in \lbrack x]$ and let $V_{y}$ be an open
neighborhood of $y$ meeting finitely many members of $\mathcal{U}$. Let $%
\mathcal{V}(y)=\{U\in \mathcal{U}: V_{y}\cap U\neq \emptyset \}$. Clearly,
since $\mathcal{U}$ consists of clopen sets and the collection $\mathcal{V}%
(y)$ is finite, the set 
\begin{equation*}
W_{y}=(V_{y}\cap \bigcap \mathcal{U}(x))\setminus \bigcup \{U\in \mathcal{V}%
(y): y\notin U\}
\end{equation*}%
is a non-empty open set of $\mathbf{X}$. Of course, $y\in W_{y}$. We show
that $W_{y}\subseteq \lbrack x]$. Fix $t\in W_{y}$. If $t\notin \lbrack x]$
then there is a $U_{t}\in \mathcal{U}$ such that $t\in U_{t}$ and $x\notin
U_{t}$. Then $y\notin U_{t}$ and $U_{t}\in \mathcal{V}(y)$. This implies
that $t\notin W_{y}$. Contradiction! Therefore, $W_{y}\subseteq \lbrack x]$, so $[x]$ is open as required. We claim that $X\backslash \lbrack x]$ is
also open. To this end, fix $z\in X\backslash \lbrack x]$. Clearly, there
exists a $U\in \mathcal{U}$ such that either $z\in U$ and $x\notin U$ or $%
z\notin U$ and $x\in U$. Assume that $U\in \mathcal{U}$ is such that $z\in U$
and $x\notin U$. In this case, it is easy to see that $U$ is a neighborhood
of $z$ included in $X\backslash \lbrack x]$. Now, assume that $U\in \mathcal{%
U}$ is such that $z\notin U$ and $x\in U$. Then $X\backslash U$ is a
neighborhood of $z$ disjoint from $[x]$. Hence, for each $x\in X$, $[x]$ is
a clopen set and, in consequence, $\{[x]: x\in X\}$ is a cellular family of
clopen sets of $\mathbf{X}$ which covers $X$. Since the mapping: 
\begin{equation*}
\lbrack x]\rightarrow \{U\in \mathcal{U}: [x]\subseteq U\}\in \lbrack 
\mathcal{U}]^{<\omega }
\end{equation*}%
is one-to-one and $|[\mathcal{U}]^{<\omega }|=\aleph _{0},$ it follows that $%
\mathbf{X}$ admits a denumerable cellular family of clopen sets as required.

We claim that $X/\sim $ is locally finite. To this end, fix $x\in X$ and let 
$V$ be an open neighborhood of $x$ meeting at most finitely many members of $%
\mathcal{U}$. Suppose that the set 
\begin{equation*}
A(V)=\{z\in X/\sim: z\cap V\neq \emptyset \}
\end{equation*}%
is infinite. For every $z\in A(V)$, let $E(z)=\bigcap\{U\in \mathcal{U}%
: z\subseteq U\}$. Since $\mathcal{U}$ is closed under finite intersections, $%
E(z)\in\mathcal{U}$ for each $z\in A(V)$. Clearly, if $z_{1},z_{2}\in A(V)$
and $z_{1}\neq z_{2}$, then $E(z_{1})\neq E(z_{2})$. This implies that the
collection $\mathcal{E}=\{E(z): z\in A(V)\}$ is infinite because $A(V)$ is
infinite. However, $V$ meets each element of $\mathcal{E}$. This is
impossible because $\mathcal{E}\subseteq\mathcal{U}$ and $V$ meets at most
finitely many members of $\mathcal{U}$. The contradiction obtained shows
that $A(V)$ is finite. Hence, $X/\sim $ is a locally finite family of clopen
subsets of $\mathbf{X}$.\smallskip

(ii) ($\rightarrow $) is straightforward.\smallskip

($\leftarrow $) Now, suppose that $\mathcal{U}$ is a denumerable locally finite family of open sets of $\mathbf{X}$. As in part (i), without loss of generality, we assume that $\mathcal{U}$ is a cover of $X$ and $\mathcal{U}$ is closed under finite
intersections. For every $x\in X$, we let 
\begin{equation*}
U_{x}=\bigcap \{U\in \mathcal{U}: x\in U\}\text{.}
\end{equation*}%
Clearly, $x\in U_{x}$ and for every $U\in \mathcal{U}$ with $%
x\in U$, the inclusion $U_{x}\subseteq U$ holds. For every $n\in \omega $, define 
\begin{equation*}
X_{n}=\{x\in X: x\text{ belongs to the boundary of at most $n$ members of }%
\mathcal{U}\}\text{.}
\end{equation*}%
Clearly, $X=\bigcup \{X_{n}: n\in \omega \}$. Working as in the proof of
Theorem 8 on p. 584 of \cite{ker}, we can show that, for each $n\in \omega$, the set $X_{n}$ is open and $X_{0}$ dense in $\mathbf{X}$. For the reader's convenience, we sketch a proof of both assertions.

We fix $n\in \omega$. To see that $X_{n}$ is open, we fix $x\in X_n$ and
a neighborhood $G\subseteq U_{x}$ of $x$ such that $G$ meets non-trivially
at most $n$ members of $\mathcal{U}$. Then $G\subseteq X_{n}$, so $X_{n}$ is
open as required.

To prove that $X_{0}$ is dense, we fix a non-empty open set $V$ of $\mathbf{X%
}$. If $V\cap U=\emptyset $ for every $U\in \mathcal{U}$, then $V\subseteq
X_{0}$. Assume that $V\cap U^{\star}\neq \emptyset $ for some $U^{\star}\in 
\mathcal{U}$ and fix $x\in V\cap U^{\star}$. Let $O$ be a neighborhood of $x$
such that $O\subseteq V\cap U^{\star}\cap U_{x}$ and the collection $%
\mathcal{U}(O)=\{U\in\mathcal{U}: O\cap U\neq\emptyset\}$ is of cardinality $%
n$ for some $n\in\mathbb{N}$. Obviously, $x\in X_n$ and, since $X_n$ is
open, the set $O^{\star}=O\cap X_n$ is a neighborhood of $x$. Let 
\begin{equation*}
W=O^{\star}\backslash \bigcup \{\overline{U}: U\in\mathcal{U}(O) \text{ and }%
x\notin U\}\text{.}
\end{equation*}%
Clearly $W\subseteq X_{0}$ and $W$ is open. It suffices to show that $W\neq
\emptyset $. Suppose that $W=\emptyset $. Then every point of $O^{\star}$
belongs to the boundary of some $U\in\mathcal{U}(O)$. Hence, 
\begin{equation*}
O^{\star}=\bigcup \{ \partial(U): U\in\mathcal{U}(O)\}\cap O^{\star}.
\end{equation*}
Since $\mathcal{U}(O)$ is finite and the sets $U$ from $\mathcal{U}$ are all
open, it follows that $\bigcup \{\partial(U): U\in\mathcal{U}(O)\}$ is
nowhere dense. Therefore, $int(\bigcup \{\partial(U): U\in\mathcal{U}%
(O)\})=\emptyset $ and, in consequence, $O^{\star}=\emptyset $.
Contradiction! Hence $X_0$ is dense.

Let us define an equivalence relation $\sim $ on $X_{0}$ by requiring: 
\begin{equation*}
x\sim y\text{ iff for every }U\in \mathcal{U},x\in U\text{ iff }y\in U.
\end{equation*}
For every $x\in X_0$ let $[x]$ denote the $\sim $ equivalence class of $x$.
Clearly, for every $x\in X_0$, $[x]\subseteq U_{x}$.

We claim that, for $x\in X_0$, the set $[x]$ is open in $\mathbf{X}$.
Clearly, for every $t\in X_0\cap( U_{x}\backslash \lbrack x]),U_{t}\subseteq U_{x}$
and $[x]\cap U_{t}=\emptyset $. We notice that if $t\in X_0\cap( U_x\setminus \lbrack
x])$ and $s\in \lbrack x]\cap \overline{U_{t}}$, then $s\in \partial(U_{t})$
and, consequently, $s\notin X_0$. Hence $[x]\cap \overline{U_{t}}%
=\emptyset$ for every $t\in X_0\cap( U_x\setminus\lbrack x])$. Since $\{\overline{U_{t}%
}: t\in X_0\cap( U_{x}\backslash \lbrack x])\}$ is a locally finite family of closed
sets, it has a closed union. Therefore, 
\begin{equation*}
\lbrack x]= X_0\cap (U_{x}\backslash \bigcup \{\overline{U_{t}}: t\in X_0\cap( U_{x}\backslash
\lbrack x])\})
\end{equation*}%
is open.

Let us prove that $X_{0}/\sim$ is infinite. Suppose that $X_{0}/\sim$ is
finite of cardinality $n\in\omega$. Let $X_{0}/\sim=\{ [x_i]: i\in n\}$.
Since $X_0$ is dense in $\mathbf{X}$, for every $U\in \mathcal{U}$, there
exists $i_{U}\in n$ such that $U\cap \lbrack x_{i_{U}}]\neq \emptyset $.
Then $x_{i_{U}}\in U$. Since $\mathcal{U}$ is infinite, it follows that
there exists $i_0\in n$, such that $x_{i_0}$ belongs to infinitely many
members of $\mathcal{U}$. This contradicts the fact that $\mathcal{U}$ is
locally finite. The contradiction obtained shows that $X_{0}/\sim$ is
infinite.

Working as in the proof of part (i), we can show that $X_0/\sim $ is locally
finite. Moreover, $\bigcup X_{0}/\sim $ is dense in $\mathbf{X}$. To
complete the proof of (ii),  it remains to check that $X_{0}/\sim $ is countable.

To show that $X_0/\sim$ is countable, it suffices to observe that the
function $H$ from $X_{0}/\sim $ \ to $[\mathcal{U}]^{<\omega }$ given by 
\begin{equation*}
H([x])=\{U\in \mathcal{U}: [x]\subseteq U\}
\end{equation*}%
is injective because if $[x],[y]\in $ $X_{0}/\sim$ and $[x]\neq \lbrack y]$,
then there exists a $U\in \mathcal{U}$ such that $x\in U$ and $y\notin U\ $%
or, $x\notin U$ and $y\in U$. In any case, $H([x])\neq H([y])$). Therefore, $%
|X_{0}/\sim |\leq |[\mathcal{U}]^{<\omega }|=\aleph _{0}$. Moreover, since $%
X_{0}/\sim $ is infinite, we infer that $|X_{0}/\sim |=\aleph _{0}$ as
required.

To prove (iii), suppose that  $\mathbf{X}$ has an infinite locally finite family $\mathcal{U}$ of open sets. Mimicking the proof to (ii), we can deduce that $\mathbf{X}$ has an infinite locally finite cellular family with a dense union. 
\end{proof}

\begin{remark}
\label{r03.2} It is obvious that if $\mathcal{C}$ is a denumerable locally
finite cellular family of clopen sets of a topological space $\mathbf{X}$
such that $\bigcup\mathcal{C}\neq X$, then $\mathcal{C}\cup\{
X\setminus\bigcup\mathcal{C}\}$ is a denumerable cellular family of clopen
sets which is a cover of $\mathbf{X}$. It is also obvious that every cover
of $\mathbf{X}$ which is a cellular family is locally finite. Hence,
condition (i) of Theorem \ref{t3.1} can be replaced with the following:

\begin{enumerate}
\item[($i^{\star}$)] $\mathbf{X}$ admits a denumerable locally finite family
of clopen sets iff $\mathbf{X}$ admits a denumerable cellular family of
clopen sets which is a cover of $\mathbf{X}$.
\end{enumerate}
\end{remark}

\begin{corollary}
\label{c03.3} The following hold in $\mathbf{ZF}$:

\begin{enumerate}
\item[(i)] For every infinite set $X$, if the Cantor cube $\mathbf{2}^{X}$
admits a denumerable locally finite family of clopen sets, then it admits a
denumerable cellular family of clopen sets which covers $2^{X}$.

\item[(ii)] The Cantor cube $\mathbf{2}^{\omega }$ admits denumerable
cellular families of clopen sets but not denumerable locally finite families
of clopen sets.
\end{enumerate}
\end{corollary}

\begin{proof}
That (i) holds follows directly from Theorem \ref{t3.1} and Remark \ref{r03.2}. To prove (ii), we notice that $\mathbf{2}^{\omega }$ is compact in $\mathbf{ZF}$, so, by Proposition \ref{p02.19},  $\mathbf{2}^{\omega}$ cannot admit denumerable locally finite families of clopen sets. However,  $\mathbf{2}^{\omega}$  admits a denumerable cellular family of clopen sets by Proposition \ref{p02.13}($d$). 
\end{proof}

\begin{theorem}
\label{t03.4} It holds in $\mathbf{ZF}$ that $\mathbf{IQDI}$ implies the
following:

\begin{enumerate}
\item[(i)] A topological space $\mathbf{X}$ admits an infinite locally
finite cellular family iff it admits a denumerable locally finite cellular
family of regular open sets.

\item[(ii)] Every topological space admitting an infinite locally finite
family of open sets admits a denumerable locally finite cellular family.

\item[(iii)] Every non-lightly compact topological space admits a
denumerable locally finite cellular family.
\end{enumerate}
\end{theorem}

\begin{proof} We assume $\mathbf{ZF+IQDI}$. To prove (i), let us suppose that  $\mathcal{P}$ is an infinite locally finite cellular
family of a topological space $\mathbf{X}$. Let, by $\mathbf{IQDI}$, $%
\{A_{n}: n\in \omega\}$ be a disjoint family of non-empty finite subsets
of $\mathcal{P}$. For every $n\in \omega$, put $O_{n}=int(\overline{%
\bigcup A_{n}})$. It is straightforward to verify that $\{O_{n}: n\in\omega\}$ is a locally finite, cellular family of regular open sets of $%
\mathbf{X}$. This completes the proof to (i).

To prove that $\mathbf{IQDI}$ implies (ii), suppose that $\mathbf{X}$ is a topological space which admits an infinite locally finite family of open sets. By Theorem \ref{t3.1} (iii), $\mathbf{X}$ admits an infinite locally finite cellular family. Therefore, it follows from (i) that $\mathbf{IQDI}$ implies that $\mathbf{X}$ admits a denumerable locally finite cellular family.

Since $\mathbf{IQDI}$ implies (ii), it follows from Theorem \ref{t02.20} that $\mathbf{IQDI}$ implies (iii).
\end{proof}

As an immediate corollary to Theorem \ref{t03.4}, we get the following
strengthening of Corollary \ref{c02.21}:

\begin{corollary}
\label{c03.5} It holds in $\mathbf{ZF}$ that $\mathbf{IQDI}$ implies that
every infinite topological space which does not admit a denumerable cellular
family is lightly compact.
\end{corollary}

In connection with Theorem \ref{t02.15} and Proposition \ref{p02.17}, it is
natural to ask the following question:

\begin{question}
\label{q03.6} Is it provable in $\mathbf{ZF}$ that, for all Hausdorff spaces 
$\mathbf{X}$ and $\mathbf{Y}$, if $\mathbf{X}\times\mathbf{Y}$ has a
denumerable cellular family, then at least one of the spaces $\mathbf{X}$
and $\mathbf{Y}$ has a denumerable cellular family?
\end{question}

A partial answer to Question \ref{q03.6} is given by the following
proposition:

\begin{proposition}
\label{p03.7} $(\mathbf{ZF})$ Let $\mathbf{X}=(X,\mathcal{T}_X)$ and $%
\mathbf{Y}=(Y,\mathcal{T}_Y)$ be topological spaces such that both $X$ and $%
Y $ are infinite. Then the following hold:

\begin{enumerate}
\item[(i)] If $\mathbf{X}$ is compact, then $\mathbf{X}\times \mathbf{Y}$
admits a denumerable locally finite family of open sets iff $\mathbf{Y}$
admits a denumerable locally finite cellular family.

\item[(ii)] If $\mathbf{X}\times\mathbf{Y}$ admits a denumerable cellular
family, then at least one of the spaces $\mathbf{X}$ and $\mathbf{Y}$ admits
a denumerable point-finite family of open sets and, in consequence, at least
one of the sets $\mathcal{T}_X$ and $\mathcal{T}_Y$ is Dedekind-infinite.

\item[(iii)] The set $X\times Y$ admits a denumerable partition iff at least
one of the sets $X$ and $Y$ is weakly Dedekind infinite.

\item[(iv)] If $\mathbf{Y}$ is discrete, then $\mathbf{X}\times \mathbf{Y}$
admits a denumerable cellular family iff at least one of the spaces $\mathbf{%
X}$ and $\mathbf{Y}$ admits a denumerable cellular family.
\end{enumerate}
\end{proposition}

\begin{proof}
(i) ($\rightarrow $)  We assume that $\mathbf{X}$ is compact. Let $\mathcal{A}=\{A_{n}: n\in \omega \}$ be a
denumerable locally finite family of open sets of $\mathbf{X}\times \mathbf{Y}$. By
Theorem \ref{t3.1}, we may assume that $\mathcal{A}$ is cellular. We show
that $\mathbf{Y}$ has a denumerable locally finite family of open sets and
then appeal to Theorem \ref{t3.1}. For $n\in \omega $, let $C_{n}$ be the
canonical projection of $A_{n}$ into $Y$. Since canonical projections are open
mappings, each $C_{n}$ is open in $\mathbf{Y}$. To show that $\mathcal{C}$ is point-finite, consider any $y\in Y$ and the set $N(y)=\{ n\in\omega: y\in C_n\}$. Suppose that $y_0\in Y$ is such that  $N(y_0)$ is infinite.  Then $\{(X\times \{y_{0}\})\cap A_{n}: n\in N(y_0)\}$ is a denumerable
locally finite cellular family of the subspace $\mathbf{X}\times \{y_{0}\}$
of $\mathbf{X}\times \mathbf{Y}$. Hence, $\mathbf{X}$ being homeomorphic to $%
\mathbf{X}\times \{y_0\},$ admits a denumerable locally finite cellular
family, contradicting the fact that $\mathbf{X}$ is compact. This is why $N(y)$ is finite for each $y\in Y$. Hence, the family $%
\mathcal{C}=\{C_{n}: n\in \omega\}$ is point-finite, so it infinite.  We claim that $%
\mathcal{C}$ is locally finite. To this end, we assume the contrary and fix $%
y^{\star}\in Y$ such that every neighborhood of $y^{\star}$ meets infinitely many members of 
$\mathcal{C}$. Let $\mathcal{U}=\{U\in \mathcal{T}_X:$ there exists $V\in \mathcal{T}_Y$ such that $%
y^{\star}\in V$ and $U\times V$ meets finitely many members of $\mathcal{A}\}$.
Since $\mathbf{X}$ is compact and $\mathcal{U}$ covers $X$, it
follows that there exist $n\in\omega$ and a subcollection $\{U_i: i\in n+1\}$ of $\mathcal{U}$  with $X=\bigcup_{i\in n+1} U_i$.
For each $i\in n+1$, we can fix an open neighborhoods $V_{i}$ of $y^{\star}$ such that $U_{i}\times V_{i}$ meets finitely many members of $\mathcal{A}$. Clearly, 
$\bigcup_{i\in n+1} (U_{i}\times V_{i})$ meets finitely many members of $%
\mathcal{A}$. The set $V=\bigcap_{i\in n+1}V_i$ is a neighborhood of $y^{\star}$. If the set $K=\{k\in \omega: V\cap C_k\neq\emptyset\}$ is infinite, then there exists $i\in n+1$ such that $U_i\times V$ meets infinitely many members of $\mathcal{A}$, contradicting our choice of the sets $U_i$ and $V_i$ for $i\in n+1$. Hence, $V$ meets only finitely many members of $\mathcal{C}$. The contradiction obtained proves that $\mathcal{C}$ is locally finite. By Theorem \ref{t3.1}, $\mathbf{Y}$ admits a
denumerable cellular locally finite family.

($\leftarrow $) \smallskip It is straightforward that if $\mathbf{Y}$ admits a denumerable locally finite cellular family, then so does $\mathbf{X}\times\mathbf{Y}$ regardless
of $\mathbf{X}$ being compact.

(ii) We assume that $\mathcal{A}=\{A_n: n\in\omega\}$ is a denumerable cellular family of  $\mathbf{X}\times\mathbf{Y}$. Suppose that $\mathbf{X}$ does not admit a denumerable point-finite family of open sets. In much the same way, as in the proof of (i), we define the family $\mathcal{C}=\{C_n: n\in\omega\}$ and, for each $y\in Y$, the set $N(y)$. If there exists $y_0\in Y$ such that $N(y_0)$ is infinite, then $\mathbf{X}$ admits a denumerable cellular family. This contradicts our assumption about $\mathbf{X}$. Hence, for each $y\in Y$, the set $N(y)$ is finite. This proves that $\mathcal{C}$ is a denumerable point-finite family of open sets of $Y$. 

(iii) ($\rightarrow $) Let $\mathcal{A}=\{A_{n}: n\in \omega \}$ be a
denumerable partition of $X\times Y$. That is, it is assumed that each $A_n$ is non-empty and $A_m\cap A_n=\emptyset$ for each pair $m,n$ of distinct elements of $\omega$; moreover, $X=\bigcup_{n\in\omega} A_n$. Assume that $X$ is not weakly Dedekind
infinite. For every $n\in \omega $, let $C_{n}$ be the canonical projection
of $A_{n}$ into $Y$ and let $\mathcal{C}=\{C_{n}: n\in \omega \}$. As in part
(i), one can prove that $\mathcal{C}$ is infinite and point-finite.  For every $y\in Y,$ let $N(y)=\{ n\in\omega: y\in C_n\}$. Since $\mathcal{A}$ is a cover of $X\times Y$, the collection $\mathcal{C}$ is a cover of $Y$. Hence,  the set $N(y)$ is non-empty for each $y\in Y$. Let $n(y)=\max N(y)$ for each $y\in Y$. For $n\in\omega$, let $B_n=\{ y\in Y: n(y)=n\}$. Clearly, the set $M=\{n\in\omega: B_n\neq\emptyset\}$ is infinite and, for each pair $m, n$ of distinct elements of $M$, $B_m\cap B_n=\emptyset$. Hence $\mathcal{P}(Y)$ is Dedekind-infinite, so $Y$ is weakly Dedekind-infinite as required.

($\leftarrow $) It is straightforward to check that if $Y$ is weakly Dedekind-infinite, then $X\times Y$ admits a denumerable partition. 

(iv) If either $\mathbf{X}$ or $\mathbf{Y}$ admits a denumerable cellular family, so does $\mathbf{X}\times \mathbf{Y}$ by Theorem \ref{t02.15}. Now, we assume that $\mathcal{A}=\{A_n: n\in\omega\}$ is a denumerable cellular family of $\mathbf{X}
\times\mathbf{Y}$. We may assume that $\mathcal{A}$ is dense in $\mathbf{X}\times\mathbf{Y}$ because, otherwise, we may add the set $(X\times Y)\setminus\overline{\bigcup\mathcal{A}}$ to $\mathcal{A}$. Suppose that $\mathbf{X}$ does not admit a denumerable cellular family.  As in parts (i)-(iii), for $n\in\omega$, let $C_n$ be the canonical projection of $A_n$ into $Y$. Since $\bigcup\mathcal{A}$ is dense in $\mathbf{X}\times\mathbf{Y}$, we have $Y=\bigcup_{n\in\omega}C_n$. In much the same way, as in the proof of (iii), we define an infinite set $M\subseteq\omega$ and a denumerable cellular family $\{B_n: n\in M\}$  of $Y$. 
\end{proof}

\begin{proposition}
\label{p03.8} $\mathbf{(ZF)}$ For every infinite topological space $\mathbf{X%
}$, the following conditions are satisfied:

\begin{enumerate}
\item[(i)] $\mathbf{X}$ admits a denumerable point-finite family of regular
open sets iff $\mathbf{X}$ admits a denumerable cellular family of regular
open sets.

\item[(ii)] If $\mathbf{X}$ admits a denumerable point-finite family of open
sets, then $\mathbf{X}$ admits a denumerable cellular family or a
point-finite tower of open sets.
\end{enumerate}
\end{proposition}

\begin{proof}
(i) ($\rightarrow $) Suppose that $\mathcal{U}$ is a  denumerable point-finite family of regular open sets of $\mathbf{X}$. As in the proof to Theorem \ref{t3.1}, without loss of generality, we may assume that $\mathcal{U}$ covers $X$ and $\mathcal{U}$ is closed under finite intersections. For every $x\in X,$ let 
\begin{equation*}
\mathcal{U}(x)=\{U\in \mathcal{U}: x\in U\}\text{ and }U_{x}=\bigcap \mathcal{%
U}(x)\text{.}
\end{equation*}%
Since $\mathcal{U}$ is point-finite and $\mathcal{U}$ is closed under finite intersections, for each $x\in X$, $U_x\in\mathcal{U}$. As in the proof to Theorem \ref{t3.1} (i), let $\sim $ be the
equivalence relation on $X$ given by: 
\begin{equation*}
x\sim y\text{ iff }U_{x}=U_{y}\text{.}
\end{equation*}%
Let $\mathcal{P}=X/\sim $ be the quotient set of $\sim $
and $\mathcal{V}=\{U_{A}: A\in \mathcal{P}\}$ where, for every $A\in \mathcal{%
P}$, $U_{A}$ is the unique element of $\mathcal{U}$ such that $U_{A}=U_{x}$
for all $x\in A$. Since the mapping $A\rightarrow U_{A}$ from $\mathcal{P%
}$ to $\mathcal{V}$ is a bijection, it follows that $|\mathcal{P}|\leq \aleph _{0}$%
. For every $A\in \mathcal{P}$, let $\mathcal{U}(A)=\{U\in\mathcal{U}: U_A\subseteq U\}$. We notice that if $A\in\mathcal{P}$, then $\mathcal{U}(A)\neq\emptyset$ because $U_A\in\mathcal{U}(A)$. Thus, since $\mathcal{U}\setminus\{\emptyset\}=\bigcup\{\mathcal{U}(A): A\in\mathcal{P}\}$, it follows that if $%
\mathcal{P}$ is finite, then there exists $A_0\in \mathcal{P}$ such that the family $\mathcal{U}(A_0)$ is infinite. This is impossible because $\mathcal{U}$ is point-finite. This proves that $\mathcal{P}$ is infinite. 

We consider the following cases:

(a) $(\mathcal{V},\subseteq )$ has infinitely many minimal elements. In this
case, $\mathcal{C}=\{V\in \mathcal{V}: V$ is minimal$\}$ is the required
denumerable cellular point-finite family of regular open sets of $\mathbf{X%
}$.

(b) $(\mathcal{V},\subseteq )$ has finitely many minimal elements. Since $\mathcal{V}$ is infinite, without
loss of generality, we may assume that $(\mathcal{V},\subseteq )$ has no
minimal elements. In this case, using the fact that $\mathcal{V}$ is denumerable, we can fix a bijection $f:\omega\to\mathcal{V}$ and construct, via a straightforward induction, a (point-finite)
tower $(V_n)_{n\in\omega}$ of $(\mathcal{V},\subseteq )$. It
follows from Proposition \ref{p2.10} that $\mathbf{X}$ has a denumerable
cellular family of regular open sets.

($\leftarrow $) This is straightforward.

(ii) This can be proved exactly as in the proof of part (i) by simply
replacing \textquotedblleft regular open\textquotedblright\ with
\textquotedblleft open\textquotedblright .
\end{proof}

\section{From regular matrices to denumerable cellular families}
\label{s4}

The following new concepts are of significant importance in the sequel:

\begin{definition}
\label{d04.1} Let $\mathbf{X}$ be a topological space. Suppose that $%
\mathcal{C}=\{\mathcal{C}_n: n\in\mathbb{N}\}$ is a collection of finite
cellular families of $\mathbf{X}$ such that $\mathcal{C}_m\neq\mathcal{C}_n$
for each pair $m,n$ of distinct natural numbers. Then $\mathcal{C}$ is
called:

\begin{enumerate}
\item[(i)] a \textit{regular matrix} of $\mathbf{X}$ if, for each $n\in%
\mathbb{N}$, $\mathcal{C}_n\subseteq RO(\mathbf{X})$ and $\bigcup\mathcal{C}%
_n$ is dense in $\mathbf{X}$;

\item[(ii)] a \textit{clopen matrix} of $\mathbf{X}$ if, for each $n\in%
\mathbb{N}$, $\mathcal{C}_n\subseteq Clop(\mathbf{X})$ and $\bigcup\mathcal{C%
}_n=X$.
\end{enumerate}
\end{definition}

\begin{remark}
\label{r04.2} Let $\mathbf{Y}$ be a regular open (resp., clopen) subspace of
a topological space $\mathbf{X}$. Suppose that $\mathcal{C}=\{\mathcal{C}%
_{n}: n\in \mathbb{N}\}$ is a regular (resp., clopen) matrix of $\mathbf{Y}$.
If $Y$ is dense in $\mathbf{X}$, then $\mathcal{C}$ is a regular (resp.,
clopen) matrix of $\mathbf{X}$. If $Y$ is not dense in $\mathbf{X}$, then,
by defining $\mathcal{C}^{\prime }=\{\mathcal{C}_{n}\cup \{X\backslash int(%
\overline{Y})\}: n\in \mathbb{N}\}$ (resp., $\mathcal{C}^{\prime }=\{\mathcal{%
C}_{n}\cup \{X\backslash Y\}: n\in \mathbb{N}\}$), we obtain a regular
(resp., clopen) matrix $\mathcal{C}^{\prime }$ of $\mathbf{X}$. On the other
hand, if $\mathcal{E}=\{\mathcal{E}_{n}: n\in \mathbb{N}\}$ is a regular
(resp., clopen) matrix of $\mathbf{X}$ and, for each $n\in \mathbb{N}$ and
each $E\in \mathcal{E}_{n}$, the set $E\cap Y$ is non-empty, then $\{\{E\cap
Y: E\in \mathcal{E}_{n}\}: n\in \mathbb{N}\}$ is a regular (resp., clopen)
matrix of $\mathbf{Y}$.
\end{remark}

\begin{lemma}
\label{l04.3} $\mathbf{(ZF)}$ For every topological space $\mathbf{X}$, the
following conditions are fulfilled:

\begin{enumerate}
\item[(i)] if $RO(\mathbf{X})$ is Dedekind-infinite, then $\mathbf{X}$
admits a regular matrix;

\item[(ii)] if $Clop(\mathbf{X})$ is Dedekind-infinite, then $\mathbf{X}$
admits a clopen matrix.
\end{enumerate}
\end{lemma}

\begin{proof}
Suppose that $RO(\mathbf{X})$ (resp., $Clop(\mathbf{X})$) is Dedekind-infinite. Then we can fix a collection $\mathcal{U}=\{ U_n: n\in\mathbb{N}\}$ such that $\mathcal{U}\subseteq RO(\mathbf{X})$ (resp., $\mathcal{U}\subseteq Clop(\mathbf{X})$ and $U_m\neq U_n$ for each pair of distinct $m,n\in\mathbb{N}$. Of course, if $U_n$ is dense in $\mathbf{X}$ for infinitely many natural numbers $n$, then $\mathbf{X}$ has a regular (resp., clopen) matrix. Therefore,  without loss of generality, we may assume that, for each $n\in\mathbb{N}$, the set $U_n$ is not dense in $\mathbf{X}$.  Put $\mathcal{E}_n=\{U_n, X\setminus\overline{U_n}\}$ for each $n\in\mathbb{N}$. It may happen that there is a pair $m, n$ of distinct natural numbers such that $U_n=X\setminus\overline{U_m}$ and, in consequence, $\mathcal{E}_n=\mathcal{E}_m$. However, we can inductively define a strictly increasing sequence $(k_n)_{n\in\mathbb{N}}$ of natural numbers such that $\mathcal{E}_{k_m}\neq\mathcal{E}_{k_{n}}$ for each pair $m,n$ of distinct natural numbers. Then putting $\mathcal{C}_n=\mathcal{E}_{k_n}$ for each $n\in\mathbb{N}$,  we obtain a regular (resp., clopen) matrix $\mathcal{C}=\{\mathcal{C}_n: n\in\mathbb{N}\}$ of $\mathbf{X}$.
\end{proof}

\begin{theorem}
\label{t04.4} $\mathbf{(ZF)}$ Let $\mathbf{X}$ be a topological space. Then $%
\mathbf{IQDI}$ implies that the following conditions are fulfilled:

\begin{enumerate}
\item[(i)] if $RO(\mathbf{X})$ is infinite, then $\mathbf{X}$ admits a
regular matrix;

\item[(ii)] if $Clop(\mathbf{X})$ is infinite, then $\mathbf{X}$ admits a
clopen matrix.
\end{enumerate}
\end{theorem}

\begin{proof} (i) Assume $\mathbf{IQDI}$. If $RO(\mathbf{X})$ is Dedekind-infinite, then $\mathbf{X}$ has a regular matrix by Lemma \ref{l04.3}. Let us suppose that $RO(\mathbf{X})$ is both infinite and Dedekind-finite. Fix, by $\mathbf{IQDI}$, a strictly ascending family $\mathcal{B}=\{%
\mathcal{B}_{n}: n\in \mathbb{N}\}$ of finite sets of regular open subsets of 
$\mathbf{X}$. For every $n\in \mathbb{N}$, let $\mathcal{G}_{n}$ be the
Boolean subalgebra of $RO(\mathbf{X})$ generated by $\mathcal{B}_{n}$. Then $\mathcal{G}=\bigcup_{n\in\mathbb{N}}\mathcal{G}_n$ is an infinite Boolean subalgebra of $RO(\mathbf{X})$.  We define
collections $\mathcal{C}_{n}$ as follows: 
\begin{equation*}
\mathcal{C}_{n}=\{C\in \mathcal{G}_{n}\backslash \{\emptyset \}: \forall G\in 
\mathcal{G}_{n}\backslash \{\emptyset \}(G\subseteq C\rightarrow G=C)\}\text{%
.}
\end{equation*}

\noindent Clearly, for each $n\in \mathbb{N}$, $\mathcal{C}_{n}\subseteq \mathcal{G}_n$ and $\mathcal{C}_n$ is a cellular family of regular open sets of $\mathbf{X}$. For each $n\in\mathbb{N}$, the collection $\mathcal{C}_n$ is the set of all atoms of the finite Boolean algebra $\mathcal{G}_n$.  If the collection  $\mathcal{C}=\{\mathcal{C}_n: n\in\mathbb{N}\}$ were finite, then $\mathcal{G}$ would be finite because every non-empty set $G\in\mathcal{G}$ is expressible as a finite union of some members of $\bigcup_{n\in\mathbb{N}}\mathcal{C}_n$. Hence $\mathcal{C}$ is infinite. Without loss of generality, we may assume that $\mathcal{C}_m\neq\mathcal{C}_n$ for each pair $m,n$ of distinct natural numbers. 

 Since, for each $n\in \mathbb{N}$, $O_{n}=int(\overline{%
\bigcup \mathcal{C}_{n}})\in RO(\mathbf{X})$ and $RO(\mathbf{X})$ is Dedekind-finite, it
follows that $\{O_{n}: n\in \mathbb{N}\}$ is finite. For our convenience we
assume that $\{O_{n}: n\in \mathbb{N}\}=\{O_{1}\}$. If $\overline{O_1}=X$, then $\mathcal{C}$ is a regular matrix of $\mathbf{X}$. So, assume that $\overline{O_1}\neq X$ and put $\mathcal{E}_n=\mathcal{C}_n\cup\{ X\setminus\overline{O_1}\}$ for each $n\in\mathbb{N}$. In this case, $\mathcal{E}=\{\mathcal{E}_n: n\in\mathbb{N}\}$ is a regular matrix of $\mathbf{X}$. 

(ii) This can be proved as in part (i) by replacing each occurrence of
regular open with clopen, and $RO(\mathbf{X})$ with $Clop(\mathbf{X})$.
 
\end{proof}

\begin{theorem}
\label{t04.5} $(\mathbf{ZF})$ Let $\mathbf{X}$ be a topological space.

\begin{enumerate}
\item[(i)] If $\mathbf{X}$ admits a regular matrix $\mathcal{C}=\{\mathcal{C}%
_{n}: n\in \mathbb{N}\}$ such that the set $D=\bigcap\{\bigcup\mathcal{C}_n:
n\in\mathbb{N}\}$ is dense in $\mathbf{X}$, then $\mathbf{X}$ admits an
infinite cellular family of regular open sets. \newline
In particular, $\mathbf{IQDI}$ implies that if $\mathbf{X}$ admits a regular
matrix $\mathcal{C}=\{\mathcal{C}_{n}: n\in \mathbb{N}\}$ such that the set $%
D=\bigcap\{\bigcup\mathcal{C}_n: n\in\mathbb{N}\}$ is dense in $\mathbf{X}$,
then $\mathbf{X}$ admits a denumerable cellular family of regular open sets.

\item[(ii)] If $\mathbf{X}$ admits a clopen matrix, then $\mathbf{X}$ admits
an infinite cellular family of clopen sets.\newline
In particular, $\mathbf{IQDI}$ implies that if $\mathbf{X}$ admits a clopen
matrix, then it admits a denumerable cellular family of clopen sets.
\end{enumerate}
\end{theorem}

\begin{proof}
(i) Let $\mathcal{C}=\{\mathcal{C}_{n}: n\in \mathbb{N}\}$ be a regular matrix of $\mathbf{X}$ such that the set $D=\bigcap
\{\bigcup \mathcal{C}_{n}: n\in \mathbb{N}\}$ is dense in $\mathbf{X}$. We are going to conclude that $\mathbf{X}$ has a denumerable cellular family of regular open sets.  It is straightforward to see
that $\mathcal{U}=\bigcup_{n\in\mathbb{N}} \mathcal{C}_n$ is infinite. Furthermore, it is easy
to see that, for every $x\in D$ and every $n\in \mathbb{N}$, $x$ belongs to a
unique element of $\mathcal{C}_{n}$. Therefore, for each $x\in D$, the set $\mathcal{U}(x)=\{U\in 
\mathcal{U}: x\in U\}$ is a countable subset of $RO(\mathbf{X})$. If, for some $x\in D$%
, $\mathcal{U}(x)$ is infinite, then the conclusion follows from Corollary %
\ref{c02.11}. Assume that, for every $x\in D$, $\mathcal{U}(x)$ is finite. Then, for
every $x\in D,U_{x}=\bigcap \mathcal{U}(x)\in RO(\mathbf{X})$. Let $\mathcal{W}%
=\{U_{x}: x\in D\}$ and check that $\mathcal{W}$ is cellular. If $x,y\in D$ are such that $%
U_{x}\neq U_{y}$ then, for some $n\in \mathbb{N}$, there exist $C,G\in 
\mathcal{C}_{n}$ such that $C\neq G$, $x\in C$ and $y\in G$. Since $C\cap
G=\emptyset $ it follows that $U_{x}\cap U_{y}=\emptyset $. Hence, $\mathcal{%
W}$ is cellular. If $\mathcal{W}$ were finite, then $\mathcal{U}$ would finite.
Contradiction! Therefore, $\mathcal{W}$ is infinite. Of course, $\mathcal{W}\subseteq RO(\mathbf{X})$.\smallskip 

The second assertion follows from the first one and Proposition \ref{p02.12}
(ii).\smallskip 

(ii) We notice that if $\mathcal{C}=\{\mathcal{C}_{n}: n\in \mathbb{N}\}$ is a clopen matrix of $\mathbf{X}$, then $X=\bigcap
\{\bigcup \mathcal{C}_{n}: n\in \mathbb{N}\}$, so, to prove (ii), we can argue in much the same way, as in the proof of part (i).
\end{proof}

\begin{theorem}
\label{t04.6} $\mathbf{(ZF})$ $\mathbf{IQDI}$ implies each of the following
statements:

\begin{enumerate}
\item[(i)] Every infinite Hausdorff space with a well-orderable dense subset
admits a denumerable cellular family.

\item[(ii)] Every infinite Hausdorff space which is also a Baire space
admits a denumerable cellular family.

\item[(iii)] Every infinite Hausdorff space such that $Clop(\mathbf{X})$ is
infinite admits a denumerable cellular family.

\item[(iv)] $\mathbf{I0}dim\mathbf{HS}(cell,\aleph _{0})$.
\end{enumerate}
\end{theorem}

\begin{proof}
Assume $\mathbf{IQDI}$ and let $\mathbf{X}=(X, \mathcal{T})$ be an infinite Hausdorff space. Then $RO(\mathbf{X})$ is infinite. If $RO(\mathbf{X})$ is Dedekind-infinite, then $\mathbf{X}$ has a denumerable cellular family by Proposition \ref{p02.12}(i). Therefore, to prove (i)-(ii), we may assume that $RO(\mathbf{X})$ is Dedekind-finite.  By Theorem \ref{t04.4}, we can fix a regular matrix $\mathcal{C}=\{\mathcal{C}_{n}: n\in \mathbb{N}\}$ of $\mathbf{X}$.  Clearly, $\mathcal{U}=\bigcup_{n\in\mathbb{N}} \mathcal{C}_n$ is
infinite. 

 To prove (i), suppose that $\mathbf{X}$ has a well-orderable  dense set $S$. Let $\leq$ be a well-ordering on $S$.  We observe that, for every cellular family $\mathcal{P}$ of $%
\mathbf{X}$, the following binary relation $\precsim $ on $\mathcal{P}$ given by: 
\begin{equation}
O\precsim Q\text{ iff }\min \{s\in S: s\in O\}\leq \min \{s\in S: s\in Q\}
\label{5}
\end{equation}%
is a well-ordering on $\mathcal{P}$. Therefore, for each $n\in\mathbb{N}$, we can fix a well-ordering $\precsim_n$ on $\mathcal{C}_n$. This implies that  $\mathcal{U}$ is countable as a countable
union of finite well-ordered sets. Since $\mathcal{U}$\ is
infinite, it follows that it is denumerable, contradicting our hypothesis on 
$RO(\mathbf{X})$. Hence (i) holds.\smallskip 

To prove (ii), assume that $\mathbf{X}$ is a Baire space. Then the set $D=\bigcap \{\bigcup \mathcal{C}%
_{n}: n\in \mathbb{N}\}$ is dense in $\mathbf{X}$. By Theorem \ref{t04.5}, $%
\mathbf{X}$ has a denumerable cellular family of regular open sets,
contradicting the assumption that $RO(\mathbf{X})$ is Dedekind finite. Hence (ii) holds.\smallskip\ 

To prove (iii), we assume that $\mathbf{X}$ is a Hausdorff space such that $Clop(\mathbf{X})$ is infinite.  By Theorem \ref{t04.4}, $\mathbf{%
X}$ admits a clopen matrix and, by Theorem \ref{t04.5}, $\mathbf{%
X}$ has a denumerable cellular family of clopen sets. Hence (iii) holds. It follows from (iii) that (iv) also holds.
\end{proof}

Now, we are in a position to give a satisfactory answer to Question \ref{q1}.

\begin{theorem}
\label{t04.7} $\mathbf{(ZF)}$ The following conditions are all equivalent:

\begin{enumerate}
\item[(i)] $\mathbf{IQDI}$;

\item[(ii)] $\mathbf{I0}dim\mathbf{HS}(cell,\aleph _{0})$;

\item[(iii)] for every infinite set $X$, every infinite subspace $\mathbf{Y}$
of the Cantor cube $\mathbf{2}^{X}$ admits a denumerable cellular family;

\item[(iv)] for every infinite set $X$, the Cantor cube $\mathbf{2}^{X}$
admits a denumerable cellular family;
\end{enumerate}
\end{theorem}

\begin{proof}
(i) $\leftrightarrow $ (iv) has been established in \cite{kt}. Since every $%
0$-dimensional Hausdorff space is homeomeorphic with a subspace of a Cantor
cube, it follows that (ii) and (iv) are equivalent. It is obvious that (ii) implies (iii) and (iii)
implies (iv). Finally, (i) $\rightarrow $ (ii) has been established in
Theorem \ref{t04.6}. \medskip 
\end{proof}

\begin{remark}
\label{r04.8} In \cite{et}, a model $\mathcal{M}$ of $\mathbf{ZF}$ was shown
in which there exists a topology $\mathcal{T}_{\omega}$ in $\omega$ such
that $(\omega, \mathcal{T}_{\omega})$ is a dense-in-itself zero-dimensional
Hausdorff space which does not admit denumerable cellular families in $%
\mathcal{M}$. Hence, each one of conditions (i)-(iv) of Theorem \ref{t04.7}
fails in $\mathcal{M}$.
\end{remark}

Let us recall the following known concept (see, for instance, \cite{hr} and 
\cite{tr}):

\begin{definition}
\label{d04.9} An infinite set $A$ is called \textit{amorphous} if, for every
infinite proper subset $B$ of $A$, the set $A\setminus B$ is finite.
\end{definition}

\begin{proposition}
\label{p04.10} The following hold in $\mathbf{ZF}$:

\begin{enumerate}
\item[(i)] If $\mathbf{X}$ is a topological space which admits a regular
matrix (resp., clopen matrix), then $RO(\mathbf{X})$ (resp, $Clop(\mathbf{X}%
) $) is quasi Dedekind-infinite.

\item[(ii)] If every infinite Hausdorff space admits a regular matrix, then
there are no amorphous sets.

\item[(iii)] If every infinite discrete space admits a clopen matrix, then
there are no amorphous sets.
\end{enumerate}
\end{proposition}

\begin{proof} Condition (i) is trivial and (ii) follows from (iii).  To prove (iii), let us suppose that there exists an amorphous set $X$. Let us consider the discrete space $\mathbf{X}=(X, \mathcal{P}(X))$. Suppose that $\mathbf{X}$ has a clopen matrix $\mathcal{C}=\{\mathcal{C}_n: n\in\mathbb{N}\}$. Since $X$ is amorphous and each $\mathcal{C}_n$ is a finite cellular family such that $X=\bigcup\mathcal{C}_n$, it follows that,
for every $n\in \mathbb{N}$, exactly one of the sets in $\mathcal{C}_{n}$ is infinite. For each $n\in\mathbb{N}$, let $G_n$ be the unique infinite set in $\mathcal{C}_n$. Let $\mathcal{G}=\{G_n: n\in\mathbb{N}\}$. Suppose that $\mathcal{G}$ is finite. Then there exists $n_0\in\mathbb{N}$ such that the set $N=\{n\in\mathbb{N}: G_{n_0}\in\mathcal{C}_n\}$ is infinite. This is impossible because the set $X\setminus G_{n_0}$ is finite, while $\mathcal{C}_m\neq\mathcal{C}_n$ for each pair of distinct elements of $N$. This shows that the collection $\mathcal{G}$ is infinite. We notice that if $A, B$ are distinct infinite subsets of $X$, then $A\cap B$ is infinite because $X$ is amorphous. Hence, since $\mathcal{G}$ is infinite, we can easily define by induction a sequence $(E_n)_ {n\in\omega}$ of infinite subsets of $X$ such that $X=E_0$ and $E_{n+1}\subset E_n$ for each $n\in\omega$. Since $X=\bigcup_{n\in\omega}(E_n\setminus E_{n+1})\cup\bigcap_{n\in\omega}E_n$, we can easily exhibit two disjoint infinite subsets of $X$. This contradicts the assumption that $X$ is amorphous.
\end{proof}

\begin{remark}
\label{r04.11} Given an infinite set $X$, the denumerable subset $\{[X]^n:
n\in\mathbb{N}\}$ of $\mathcal{P}(\mathcal{P}(X))$ witnesses that $\mathcal{P%
}(X)$ is weakly Dedekind-infinite in $\mathbf{ZF}$. If the discrete space $%
\mathbf{X}=(X, \mathcal{P}(X))$ admits a clopen matrix $\{\mathcal{C}_n: n\in%
\mathbb{N}\}$, then the denumerable subset $\{\mathcal{C}_n: n\in\mathbb{N}%
\} $ of $[\mathcal{P}(X)]^{<\omega}$ witnesses that $\mathcal{P}(X)$ is
quasi Dedekind-infinite.
\end{remark}

We recall that $\mathbf{IQDI}(\mathcal{P})$ states that, for every infinite
set $X$, $\mathcal{P}(X)$ is quasi Dedekind-infinite (see Section 1).

\begin{proposition}
\label{p4.12} It holds in $\mathbf{ZF}$ that $\mathbf{IQDI}(\mathcal{P})$ is
equivalent to: Every infinite discrete space admits a clopen matrix.
\end{proposition}

\begin{proof}
Let $X$ be an infinite set. Suppose that $\mathcal{P}(X)$ is quasi Dedekind-infinite. Let $\{\mathcal{A}_n: n\in\omega\}$ be a denumerable set of pairwise distinct elements of $[\mathcal{P}(X)]^{<\omega}$ and, for $n\in\omega$, let $\mathcal{B}_n=\bigcup_{i\in n+1}\mathcal{A}_i$. We may assume that $\mathcal{B}_n\neq\mathcal{B}_{n+1}$ for each $n\in\omega$. We can mimic the proof to Theorem \ref{t04.4} to deduce that the discrete space $\mathbf{X}=(X, \mathcal{P}(X))$ admits a clopen matrix. This, together with Remark \ref{r04.11}, completes the proof.
\end{proof}

\begin{corollary}
\label{c04.13} $\mathbf{IQDI}(\mathcal{P})$ is not provable in $\mathbf{ZF}$%
. More precisely, $\mathbf{IQDI}(\mathcal{P})$ fails in every model of $%
\mathbf{ZFA}$ in which there are amorphous sets.
\end{corollary}

\begin{proof}  Let $\mathcal{M}$ be any model of $\mathbf{ZFA}+\neg E(I, Ia)$. For
instance, model $\mathcal{M}$37 of \cite{hr} is a model $\mathbf{ZF}+\neg E(I, Ia)$. It follows directly from Propositions \ref{p04.10} and \ref{p4.12} that $\mathbf{IQDI}(\mathcal{P})$ fails in $\mathcal{M}$.
\end{proof}

\begin{proposition}
\label{p04.12} $\mathbf{(ZF)}$ Let $\mathbf{X}$ be a Hausdorff space which
has a dense well-orderable subset. Then the following conditions are
satisfied:

\begin{enumerate}
\item[(i)] $\mathbf{X}$ admits a regular matrix if and only if $RO(\mathbf{X}%
)$ is Dedekind-infinite;

\item[(ii)] if $\mathbf{X}$ is zero-dimensional, then $\mathbf{X}$ admits a
clopen matrix if and only if $Clop(\mathbf{X})$ is Dedekind-infinite.
\end{enumerate}
\end{proposition}

\begin{proof} Let $S$ be a well-orderable dense set in $\mathbf{X}$. In the light of Lemma \ref{l04.3}, to prove (i), it suffices to check that if $\mathbf{X}$ admits a regular matrix, then $RO(\mathbf{X})$ is Dedekind-infinite. So, suppose that $\mathcal{C}=\{\mathcal{C}_n: n\in\mathbb{N}\}$ is a regular matrix of $X$. Let $\mathcal{U}=\bigcup_{n\in\mathbb{N}}\mathcal{C}_n$. In much the same way, as in the proof to Theorem \ref{t04.6}(i), one can show that $\mathcal{U}$ is denumerable, so $RO(\mathbf{X})$ is Dedekind-infinite. The proof to (ii) is similar.
\end{proof}

We are unable to solve the following problem:

\begin{problem}
\label{q04.13} Is it provable in $\mathbf{ZF}$ that, for every infinite
Hausdorff space $\mathbf{X}$ which admits a regular (resp., clopen) matrix,
the set $RO(\mathbf{X})$ (resp., $Clop(\mathbf{X})$) is Dedekind-infinite?
\end{problem}

\begin{remark}
\label{r04.14} Let us notice that if the answer to Problem \ref{q04.13} is
in the affirmative, then so is the answer to Problem \ref{q2}. Indeed,
assume that $\mathbf{IQDI}$ holds and assume that, for every infinite
Hausdorff space $\mathbf{X}$ which admits a regular matrix, $RO(\mathbf{X})$
is Dedekind-infinite. Consider any infinite Hausdorff space $\mathbf{X}$. In
view of Proposition \ref{p02.12}(i), to prove that $\mathbf{X}$ admits a
denumerable cellular family, it suffices to show that $RO(\mathbf{X})$ is
Dedekind-infinite. Since $RO(\mathbf{X})$ is infinite, it follows from
Theorem \ref{t04.4} that $X$ admits a regular matrix. Hence, by our
assumption, $RO(\mathbf{X})$ is Dedekind-infinite.
\end{remark}

\begin{theorem}
\label{t04.15}

\begin{enumerate}
\item[(a)] $\mathbf{(ZF)}$ The following are equivalent:

\begin{enumerate}
\item[(i)] $\mathbf{CAC}(\mathbb{R})$;

\item[(ii)] for every infinite second-countable Hausdorff space $\mathbf{X}$%
, it holds that every base of $\mathbf{X}$ admits a denumerable cellular
family;

\item[(iii)] for every infinite second-countable metrizable space $\mathbf{X}
$, it holds that every base of $\mathbf{X}$ admits a denumerable cellular
family;

\item[(iv)] every base of the Cantor cube $\mathbf{2}^{\omega }$ admits a
denumerable cellular family;

\item[(v)] every base of the real line $\mathbb{R}$ with the natural
topology contains a denumerable cellular family of $\mathbb{R}$.
\end{enumerate}

In particular, $\mathbf{IDI}$ is relatively consistent with $\mathbf{ZF}$
and the negation of the sentence ``for every infinite Hausdorff space $%
\mathbf{X}$, every base of $\mathbf{X}$ contains a denumerable cellular
family of $\mathbf{X}$''.

\item[(b)] It holds in $\mathbf{ZF}$ that the statement \textquotedblleft
for every set $X$, every base of the Cantor cube $\mathbf{2}^{X}$ admits a
denumerable cellular family\textquotedblright\ implies $\mathbf{IDI}$.
\end{enumerate}
\end{theorem}

\begin{proof} 
(i) $\rightarrow $ (ii) Fix a second-countable Hausdorff space $\mathbf{X}$ and a base $\mathcal{B}$ of $\mathbf{X}$. Let $\mathcal{H}=\{ H_i: i\in\omega\}$ be a countable
base of $\mathbf{X}$. Via a straightforward
induction we construct a denumerable cellular family $\mathcal{C}%
=\{C_{n}: n\in \omega\}\subseteq \mathcal{H}$ as follows.

Let $k_0=\min\{i\in\omega: X\setminus\overline{H_i}\neq\emptyset\}$ and $C_0=H_{k_0}$. Suppose that $n\in\omega$ is such that, for each $i\in n+1$,  we have already defined $k_i\in\omega$ such that, for $C_i=H_{k_i}$, the set $X\setminus\overline{\bigcup_{i\in n+1}C_i}\neq\emptyset$. We terminate the induction by putting 
\begin{equation*}
k_{n+1}=\min\{ j\in\omega\setminus\{k_i: i\in n+1\}: X\setminus\overline{\bigcup_{i\in n+1}C_i\cup H_j}\neq\emptyset\}
\end{equation*}
and $C_{n+1}=H_{k_{n+1}}$.

For every $n\in \omega$, let $\mathcal{A}_{n}=\{B\in \mathcal{B}: B\subseteq
C_{n}\} $. Since $\mathcal{P}(\mathcal{H})$ is equipotent with $\mathbb{R}$, $\mathcal{B}$ is equipotent with a subset of $\mathbb{R}$.  Hence, by $\mathbf{CAC}(\mathbb{R})$, there exists a function $\psi\in\prod_{n\in\omega}\mathcal{A}_n$.   Clearly, $\{ \psi(n): n\in\omega\}$ is a denumerable cellular family contained in $\mathcal{B}$.

(ii) $\rightarrow $ (iii) and (iii) $\rightarrow $ (iv) are
straightforward.\smallskip

(iv) $\rightarrow $ (i) Fix a disjoint family $\mathcal{A}%
=\{A_{n}: n\in \mathbb{N}\}$ of non-empty subsets of $\mathbb{R}$. We assume that (iv) holds and show
that $\mathcal{A}$ has a partial choice function. Since the sets $\mathbb{R}$ and $2^{\omega
}$ are equipotent, we may assume that $A_{n}\subseteq 2^{\omega }$ for every $n\in \mathbb{N}$.
For every $n\in \mathbb{N}$, $p\in 2^{n}$ and $x\in 2^{\omega}$, let $z_{p,x}\in 2^{\omega}$ be the function given by the rule: 
\begin{equation*}
z_{p,x}(i)=\left\{ 
\begin{array}{c}
p(i)\text{ if }i\in n, \\ 
x(i)\text{ if }i\in\omega\setminus n%
\end{array}%
\right. . 
\end{equation*}%
For $n\in\mathbb{N}$ and $p\in 2^n$, we define 
\begin{equation*}
A_{n,p}=\{z_{p,x}: x\in A_{n}\}.
\end{equation*}%
It is straightforward to verify that 
\begin{equation*}
\mathcal{B}^{\prime }(\omega )=\{[p]\setminus \{z\}: p\in 2^{n}, z\in
A_{n,p},n\in \mathbb{N}\}
\end{equation*}%
is a base for $\mathbf{2}^{\omega }$. Let, by our hypothesis,  $\mathcal{C}=\{C_n: n\in\mathbb{N}\}$ be a denumerable cellular family contained in $%
\mathcal{B}^{\prime }(\omega )$. For each  $n\in\mathbb{N}$, there exist a unique $k_n\in\mathbb{N}$, a unique $p_n\in 2^{k_n}$ and a unique $z_n\in A_{k_n, p_n}$ such that $C_n=[p_n]\setminus \{ z_n\}$. Since $\mathbf{2}^{\omega}$ is dense-in-itself, it follows that if $m,n\in\mathbb{N}$ and $m\neq n$, then $[p_n]\cap [p_m]=\emptyset$ because $C_m\cap C_n=\emptyset$. Therefore, $z_m\neq z_n$ for distinct $m,n\in\mathbb{N}$ and, moreover, the set $\{k_n: n\in\mathbb{N}\}$ is infinite. There is a strictly increasing subsequence of the sequence $(k_n)_{n\in\mathbb{N}}$, so, without loss of generality, we may assume that $k_{n}<k_{n+1}$ for each $n\in\mathbb{N}$.  Hence, $z_{n}\in
A_{k_{n},p_{n}}$. For $n\in\mathbb{N}$, let $Z_n=\{x\in 2^{\omega}: z_n=z_{p_n, x}\}$. For each $n\in\mathbb{N}$,  the set $Z_n$ is finite and we can fix a well-ordering $\leq_n$ of  $Z_n$; furthermore $A_{k_n}\cap Z_n\neq\emptyset$. Now, we can definie a partial choice function of $\mathcal{A}$ as follows:  for $n\in\mathbb{N}$, let $f(k_n)$ be the first element of $(A_{k_n}\cap Z_n, {\leq}_n)$. In this way, we have proved that conditions (i)-(iv) are all equivalent. Of course (v) follows from (ii). We show below that (v) implies (i).

(v) $\rightarrow $ (i) Assume that every base of $\mathbb{R}$ contains a denumerable cellular family. In the light of Theorem 3.14 of \cite{kw}, $\mathbf{CAC}(\mathbb{R})$ and $\mathbf{CAC}_D(\mathbb{R})$ are equivalent. Hence, to prove that $\mathbf{CAC}(\mathbb{R})$ holds, it suffices to show that every denumerable disjoint family of dense subsets of $\mathbb{R}$ has a choice function.

 Fix a disjoint family $\mathcal{A}=\{A_i: i\in\omega\}$ of dense subsets of $\mathbb{R}$. Obviously, we assume that $A_i\cap A_j=\emptyset$ for each pair of distinct element $i,j$ of $\omega$.  Let $\mathcal{Q}=\{Q_n: n\in\mathbb{N}\}$ be the family of all open intervals of $\mathbb{R}$ with rational endpoints enumerated in such a way that $Q_m\neq Q_n$ for each pair $m, n$ of distinct  natural numbers. For every $n\in\omega$, let
$$\mathcal{W}_n=\{ Q_{n+1}\setminus F: F\in [\mathbb{R}]^{n+1}, \vert F\cap A_{i}\vert =1\text{ for all } i\in n+1\}.$$
It is easy to verify that $\mathcal{W}=\bigcup_{n\in\omega}\mathcal{W}_n$ is a base of $\mathbb{R}$ with the usual topology. Let, by our hypothesis, $\mathcal{C}=\{C_n: n\in\mathbb{N}\}$ be a denumerable cellular family contained in $\mathcal{W}$. It is straightforward to check that, for every $n\in\mathbb{N}$, there exist a unique $k_n\in\omega$ and a unique $F_n\in [\mathbb{R}]^{k_n+1}$, such that  $C_n=Q_{k_n+1}\setminus F_n$ and $\vert F_n\cap A_i\vert=1$ for each $i\in k_n+1$. For our convenience, we may assume that the sequence $(k_n)_{n\in\omega}$ is strictly increasing. Now, we can define a choice function $f$ of $\mathcal{A}$ as follows: if $i\in k_0+1$, let $f(i)$ be the unique element of $F_0\cap A_i$ and, if $n\in\omega$, then, for $j\in (k_{n+1}+1)\setminus (k_n+1)$, let $f(j)$ be the unique element of $F_{n+1}\cap A_j$. This completes the proof that (v) implies (i).

The second assertion of ($a$) follows from the fact that $\mathbf{IDI}$ holds but $%
\mathbf{CAC}(\mathbb{R})$ fails in Sageev's Model I, that is, in model $\mathcal{M}6$ in 
\cite{hr}. This completes the proof of ($a$).\smallskip

($b$) Assume that, for every infinite set $X$, every base of $\mathbf{2}^X$ contains a denumerable cellular family of $\mathbf{2}^X$.  Fix an infinite set $X$. By Theorem \ref{t04.7} there exists a 
disjoint family $\mathcal{A}=\{A_{n}: n\in \mathbb{N}\}$ of finite non-empty
subsets of $X$. Let $Y=\bigcup \mathcal{A}$. We show that 
$Y$ is Dedekind-infinite. For every $n\in \omega$, let 
\begin{equation*}
\mathcal{B}_{n}=\{[p]\backslash \{h\}: p\in 2^{\cup \{A_{i}: i\in n+1\}},h\in
\lbrack p]\text{ and }|h^{-1}(1)\cap A_{n+1}|=1\}\text{. }
\end{equation*}%
Let $\mathcal{B}=\{\mathcal{B}_{n}: n\in \omega\}$. It is straightforward to verify that $\mathcal{B}$ is a base for $\mathbf{2}^{Y}$. By
our hypothesis, $\mathcal{B}$ admits a denumerable cellular family of $\mathbf{2}^Y$. Let $\mathcal{C}=\{C_n: n\in\omega\}$ be a subfamily of non-empty subsets of $\mathcal{B}$ such that $C_m\cap C_n=\emptyset$ for each pair of distinct $m,n\in\omega$. Arguing in much the same way, as in the proof that (iv) implies (i), we can show that there exists a strictly increasing sequence $(k_n)_{n\in\mathbb{N}}$ of natural numbers such that, for each $n\in\mathbb{N}$, there exist a unique $p_{n}\in 2^{\cup \{A_{i}: i\in k_{n}+1\}}$ and a unique $h_n\in [p_n]$ such that $|h_n^{-1}(1)\cap A_{k_n+1}|=1$ and $C_n=[p_n]\setminus\{ h_n\}$. For each $n\in\mathbb{N}$,  let $y_{n}$ be the unique element of $A_{k_{n}+1}$ with $%
h_n(y_n)=1$. Clearly, $\{y_{n}: n\in \mathbb{N}\}$ is a denumerable subset of 
$Y$, so $Y$ is Dedekind-infinite as required.
\end{proof}

\begin{remark}
\label{r04.16} Let $C$ be the Cantor ternary set in the unit interval $[0,1]$
of the real line $\mathbb{R}$. It is known that $\mathbf{2}^{\omega }$ and $%
\mathbf{C}$ are homeomorphic where $\mathbf{C}$ is $C$ equipped with the
usual topology inherited from $\mathbb{R}$. Let $g:[0,1]\rightarrow \lbrack
0,1]$ be the Cantor continuous increasing function such that $g(C)=[0,1]$.
By using $g$, together with the fact that $\mathbf{CAC}(\mathbb{R})$ and $%
\mathbf{CAC}_{D}(\mathbb{R})$ are equivalent (see Theorem 3.14 of \cite{kw}%
), one can easily prove that $\mathbf{CAC}(\mathbb{R})$ is equivalent to the
following sentence: Every denumerable disjoint family of dense subsets of $%
\mathbf{2}^{\omega }$ has a choice function. Then, arguing similarly to the
proof that (v) implies (i) in Theorem \ref{t04.15}, one can show that (iv)
implies (i) in Theorem \ref{t04.15}.
\end{remark}

\section{Pseudocompactness and paracompactness of Cantor cubes}
\label{s5}

Considering Proposition \ref{p02.22} and Theorem \ref{t04.7}, one may ask
whether it is consistent with $\mathbf{ZF}$ that there exist
non-pseudocompact Cantor cubes. A positive answer to this question can be
deduced from the following theorem:

\begin{theorem}
\label{t05.1} $\mathbf{(ZF)}$ Each one of the following sentences implies $%
\mathbf{CAC}_{fin}$:

\begin{enumerate}
\item[(i)] For every infinite set $X$, $\mathbf{2}^X$ is lightly compact.

\item[(ii)] For every infinite set $X$, $\mathbf{2}^X$ is pseudocompact.

\item[(iii)] For every infinite set $X$, $\mathbf{2}^X$ satisfies conditions 
$(B_1)-(B_5)$ of Proposition \ref{p02.19}.
\end{enumerate}

In consequence, none of the above statements (i)-(iii) is a theorem of $%
\mathbf{ZF}$.
\end{theorem}

\begin{proof}
To prove that (iii) implies $\mathbf{CAC}_{fin}$, let us  assume (iii) and fix a disjoint family $%
\mathcal{A}=\{A_{n}: n\in \omega \}$ of non-empty finite sets. It suffices to show that $\mathcal{A}$ has a partial
choice function. Assume the contrary and put $%
X=\bigcup \mathcal{A}$. For every $n\in \omega$, let 
\begin{equation*}
H_{n}=\{f\in 2^{X}: \text{for each }i\in n+1,|f^{-1}(1)\cap A_{i}|=1\}\text{.}
\end{equation*}%
To check that, for every $n\in \omega$, $H_{n}$ is open in $\mathbf{2}^{X}$, we notice that if $n\in\omega$ and $h\in
H_{n},$ then $|h^{-1}(1)\cap A_{i}|=1$ for each $i\in n+1$. Hence, for $p=h|\bigcup_{i\in n+1} A_{i}$ and every $g\in \lbrack p],|g^{-1}(1)\cap A_{i}|=1$ if $i\in n+1$. Therefore, $[p]\subseteq H_{n}$ and, in consequence,  $H_{n}$ is open. Using similar arguments, one can
show that, for every $n\in \omega$, $H_{n}^{c}=2^X\setminus H_n$ is open. Hence each $%
H_{n} $ is a clopen subset of $\mathbf{2}^{X}$. 

We claim that the collection $\mathcal{H}%
=\{H_{n}: n\in \omega \}$ is locally finite. To this end, fix $f\in 2^{X}$.
Since $\mathcal{A}$ has no partial choice function, it follows that there exists $%
n_{f}\in \omega$ such that for every $n\in \omega\setminus n_f$,  $|f^{-1}(1)\cap A_{n}|\neq 1$. Hence, $[f|A_{n_{f}}]$ is a
neigborhood of $f$ avoiding each $H_{n}$ for $n\in\omega\setminus n_{f}$. Thus, $\mathcal{H}$ is
locally finite as claimed. On the other hand, by (iii), $\mathbf{2}^{X}$ has no
infinite locally finite families of open sets. This contradiction shows that 
$\mathcal{A}$ has a partial choice function as required.\smallskip

That (i) implies $\mathbf{CAC}_{fin}$ follows from the observation that (iii) implies $\mathbf{CAC}_{fin}$ and
(i) implies (iii) by Theorem \ref{t02.20}.\smallskip

To prove that (ii) implies $\mathbf{CAC}_{fin}$, we fix $\mathcal{A}$ and $\mathcal{H}$
as in the proof of (iii) $\rightarrow $ $\mathbf{CAC}_{fin}$. Suppose that $%
\mathcal{A}$ does not have a partial choice function. As we have observed in the proof of (iii) $%
\rightarrow $ $\mathbf{CAC}_{fin}$, for every $n\in\omega$, the set $H_{n}$ is clopen in $\mathbf{2}^X$. Since $\mathcal{H}$ is locally finite, it follows from Corollary  \ref%
{c03.3}(i) that $\mathbf{2}^{X}$ admits a denumerable cellular family $%
\mathcal{C}=\{C_{n}: n\in \omega\}$ of clopen sets which is a cover of $\mathbf{2}^X$. 
 Define a function $f: \mathbf{2}^{X}\rightarrow \mathbb{R}$ by
requiring: $f(x)=n$ for each $n\in\omega$ and each $x\in C_n$. Since $\mathcal{C}$ is a disjoint collection
of clopen subsets of $2^{X}$ and the restriction of $f$ to each member of
the collection is continuous, it follows that $f$ is continuous. Since $f$
is unbounded, $\mathbf{2}^{X}$ is not pseudocompact. Therefore  (ii) implies $\mathbf{CAC}_{fin}$.
\end{proof}

\begin{corollary}
\label{c05.2} In every model of $\mathbf{ZF}+\neg\mathbf{CAC}_{fin}$, there
exist Cantor cubes that are not pseudocompact. In particular, in Pincus'
Model I ($\mathcal{M}4$ in \cite{hr}) and in Cohen's Second Model ($\mathcal{%
M}7$ of \cite{hr}), there exist Cantor cubes that are not pseudocompact.
\end{corollary}

To avoid misunderstanding, let us recall the following definition:

\begin{definition}
\label{d05.3} A topological space $\mathbf{X}$ is called (\textit{countably}%
) \textit{paracompact} if every (countable) open cover of $\mathbf{X}$ has a
locally finite open refinement.
\end{definition}

We obtain the following new results by applying Theorems \ref{t3.1} and \ref%
{t04.7}:

\begin{theorem}
\label{t05.4} $\mathbf{(ZF)}$ If every Cantor cube is countably paracompact,
then $\mathbf{IQDI}$ holds.
\end{theorem}

\begin{proof}
Let us suppose that all Cantor cubes are countably paracompact. In view of Theorem \ref{t04.7}, to prove that $\mathbf{IQDI}$ holds, it suffices to show that, for every infinite
set $X$, the Cantor cube $\mathbf{2}^{X}$ has a denumerable cellular family. Let us fix an
infinite set $X$. Clearly, for every $n\in \omega$, the set 
\begin{equation*}
D_{n}=\{f\in 2^{X}: |f^{-1}(1)|\geq n+1\}
\end{equation*}%
is open in $\mathbf{2}^{X}$. Since $\mathbf{2}^{X}$ is countably paracompact, there exists a locally finite open refinement $\mathcal{W}$  of the open cover $\mathcal{D}=\{D_{i}: i\in
\omega \}$. We define a collection $\mathcal{H}=\{ H_i: i\in\omega\}$ as follows:
\begin{equation*}
H_{0}=\bigcup \{W\in \mathcal{W}: W\subseteq D_{0}\}
\end{equation*}%
\noindent and, for each $k\in\omega$,  
\begin{equation*}
H_{k+1}=\bigcup \{W\in \mathcal{W}: W\subseteq D_{k+1}\text{ and }W\nsubseteq
D_{i}\text{ where } i\in k+1\}\text{.}
\end{equation*}%
Clearly, $\mathcal{H}$ is a locally finite open cover
of $\mathbf{2}^{X}$. Hence, by Theorem \ref{t3.1}, $\mathbf{2}^{X}$ admits a denumerable
(locally finite) cellular family. This completes the proof of the
theorem.
\end{proof}

\begin{corollary}
\label{c05.5}

\begin{enumerate}
\item[(i)] In every model of $\mathbf{ZF+\neg IQDI}$, there exist Cantor
cubes that are not countably paracompact.

\item[(ii)] If $\mathcal{M}$ is a model of $\mathbf{ZF}$ in which all Cantor
cubes are countably paracompact, then $\mathbf{I0}dim\mathbf{HS}(cell,
\aleph_0)$ holds in $\mathcal{M}$.
\end{enumerate}
\end{corollary}

\begin{remark}
\label{r05.6} We recall that $\mathbf{MP}$ states that all metrizable spaces
are paracompact (see the list of forms in Section 1). In \cite{gtw}, it was
proved, by a forcing argument, that $\mathbf{MP}$ is not a theorem of $%
\mathbf{ZF}$. In fact, it was shown in \cite{gtw} that even the Principle of
Dependent Choices ($\mathbf{DC}$) (see Form 43 in \cite{hr} ) does not imply 
$\mathbf{MP}$. To the best of our knowledge, it is unknown whether $\mathbf{%
MP}$ implies any weak form of the axiom of choice mentioned in Section 1.
\end{remark}

In view of Theorem \ref{t05.4} and Proposition 2.3 of \cite{ew}, one may
suspect that it is relatively consistent with $\mathbf{ZF}$ the existence of
an infinite set $X$ such that the Cantor cube $\mathbf{2}^X$ is metrizable
and not paracompact. However, we can state the following theorem:

\begin{theorem}
\label{t05.7} $\mathbf{(ZF)}$ Every metrizable Cantor cube is paracompact.
\end{theorem}

\begin{proof}
 Let $X$ be an infinite set such that the Cantor cube $\mathbf{2}^X$ is metrizable. By Theorem 2.2 of \cite{ew}, X can be expressed as the union of a strictly ascending family $\{A_n: n\in\omega\}$ of non-empty finite subsets of X. For every $n\in\omega$, let 
$\mathcal{B}_n=\{[p]: p\in 2^{A_n}\}$. Then $\mathcal{B}=\bigcup_{n\in\omega}\mathcal{B}_n$ is a $\sigma$-locally finite base of $\mathbf{2}^X$. By Theorem 2 of \cite{hkrs}, $\mathbf{2}^X$ is paracompact.

\end{proof}

\section{$\mathbf{IDI}_F$ and towers of infinite Boolean algebras}
\label{s6}

\begin{definition}
\label{d06.1} Let $\mathcal{A}$ be a collection of finite sets such that $|%
\mathcal{A}|\geq 2$. Then:

\begin{enumerate}
\item[(i)] a finite set $r$ such that $r=x\cap y$ for each pair $x,y$ of
distinct sets from $\mathcal{A}$ is called a \textit{root} of $\mathcal{A}$;

\item[(ii)] $\mathcal{A}$ is called a $\Delta $\textit{-system} if it has a
root.
\end{enumerate}
\end{definition}

The following lemma concerning $\Delta $\textit{-}systems is well known. We
show that its proof can be given in $\mathbf{ZF}$.

\begin{lemma}
\label{l06.2} $(\mathbf{ZF})$ For a fixed $k\in\mathbb{N}$, let $\mathcal{A}$
be a denumerable family of $k$-sized sets. Then, there exists an infinite
subcollection $\mathcal{B}$ of $\mathcal{A}$ such that $\mathcal{B}$ is a $%
\Delta$-system with a root $r$.
\end{lemma}

\begin{proof}
We prove the lemma by induction with respect to $k$.

For $k=1$ simply take $r=\emptyset $.

Assume that the lemma is true for every $k<n$ and let  $\mathcal{A}$  be a denumerable family of $n$-sized sets. We assume that $\mathcal{A}=\{ A_i: i\in\omega\}$ and $A_i\neq A_j$ for each pair $i,j$ of distinct elements of $\omega$. We define an
equivalence relation $\sim $ on $\mathcal{A}$ by requiring $A\sim B$ iff
there exist $v\in\omega$ and a collection $\{S_{i}: i\in v+1\}\subseteq \mathcal{A}$ such that $A\cap
S_{0}\neq \emptyset$, $S_{i}\cap S_{i+1}\neq\emptyset$ for each $i\in v+1$, and  $S_{v}\cap B\neq \emptyset $. For every $A\in \mathcal{A}
$, let $[A]$ denote the $\sim $ equivalence class of $A$. Clearly, for
every pair $A,B\in \mathcal{A}$, if $[A]\neq \lbrack B]$, then $A\cap B=\emptyset $%
. We consider the following two cases:\smallskip

i) The quotient $\mathcal{A}/\sim $ is infinite. For $A\in\mathcal{A}$, let $n([A])=\min\{ i\in\omega: A_i\in [A]\}$ and let $C_{[A]}=A_{n([A])}$. Evidently, $%
\mathcal{B}=\{C_{[A]}: A\in \mathcal{A}\}$ is an infinite disjoint
family of members of $\mathcal{A}$, and we can let $r=\emptyset $.\smallskip

ii) $\mathcal{A}/\sim $ is finite. Fix $A\in \mathcal{A}$ with $[A]$
infinite. We consider the following two subcases:\smallskip

ii) (a) There exists a finite subset $F$ of $\bigcup [A]$ such that, for
every $G\in \lbrack A],G\cap F\neq \emptyset $. Since $F$ is finite, it
follows that there exists a subset $S$ of $F$ and an infinite subfamily $%
\mathcal{B}$ of $[A]$ such that, for every $B\in \mathcal{B},B\cap F=S$.
Clearly, $\mathcal{A}^{\prime }=\{B\backslash S: B\in \mathcal{B}\}$ is a
denumerable family of $n-|S|$ sized sets. So, by our induction hypothesis,
there exists an infinite subfamily $\mathcal{B}^{\prime }$ of $\mathcal{A}%
^{\prime }$  such that $\mathcal{B}^{\prime}$ is a $\Delta$-system with a root $t$. Clearly, $r=t\cup S$ is a root of the infinite
subfamily $\{B\cup S: B\in \mathcal{B}^{\prime }\}$ of $\mathcal{A}$%
.\smallskip

ii) (b) For every finite subset $F$ of $\bigcup [A]$ there is a $G\in
\lbrack A]\ $with $G\cap F=\emptyset $. In this case we construct via an
easy induction a denumerable disjoint subfamily $\mathcal{B}%
=\{B_{n}: n\in \omega\}$ of $[A]$. Let $N([A])=\{i\in\omega: A_i\in [A]\}$. For $n=0$ we let $B_{0}=A_{\min N([A])}$. Now, assume that $n\in\omega$ is such that we have defined a disjoint subfamily $\{B_i: i\in n+1\}$ of $\lbrack A]$. Since $F=\bigcup_{i\in n+1} B_i$ is finite, by our hypothesis, some member of $[A]$ is disjoint from $F$. Let $B_{n+1}=A_{j(n)}$ where $j(n)=\min\{ i\in\omega: A_i\in [A]\text{ and } A_i\cap F=\emptyset\}$. Clearly, $r=\emptyset $ is a root of $%
\mathcal{B}$, terminating the proof of ii) (b) and the proof of the
lemma.\medskip
\end{proof}

\begin{theorem}
\label{t06.3} $\mathbf{(ZF)}$

\begin{enumerate}
\item[(i)] For every pair of numbers $m,k\in \mathbb{N}$ such that $k<m$, $%
\mathbf{IDI}_{k}$ implies $\mathbf{IDI}_{m}$. In particular, for every
natural number $m\geq 2$, $\mathbf{IDI}_{2}$ implies $\mathbf{IDI}_{m}$.

\item[(ii)] $\mathbf{CMC}$ implies $\mathbf{IQDI}$.

\item[(iii)] $\mathbf{IDI}_{F}$ implies $\mathbf{IQDI}$.

\item[(iv)] $\mathbf{IDI}$ implies $\mathbf{IDI}_{2}$ but there is a $%
\mathbf{ZF}$\ model $\mathcal{M}$ including a Dedekind-finite set $X$ such
that $[X]^{2}$ is Dedekind-infinite in $\mathcal{M}$.

\item[(v)] The following are equivalent:

\begin{enumerate}
\item[(a)] $\mathbf{IQDI}$;

\item[(b)] for every infinite set $X$, the poset $([X]^{<\omega },\subseteq
) $ has a denumerable antichain;

\item[(c)] for every infinite set $X$, the poset $([X]^{<\omega },\subseteq
) $ has a tower.
\end{enumerate}

\item[(vi)] \cite{kk} The following are equivalent:

\begin{enumerate}
\item[(d)] $\mathbf{IWDI}$;

\item[(e)] for every infinite set $X$, $X$ has a denumerable partition into
infinite sets, i.e., the poset $(\mathcal{P}(X),\subseteq )$ has a
denumerable antichain;

\item[(f)] for every infinite set $X$, the poset $(\mathcal{P}(X),\subseteq
) $ has tower.
\end{enumerate}
\end{enumerate}
\end{theorem}

\begin{proof}
(i) Fix a set $X$ such that $[X]^{k}$ is Dedekind-infinite for some $k\in 
\mathbb{N}$ and let $m\in\mathbb{N}$ be such that $m>k$. We show that $%
[X]^{m}$ is Dedekind-infinite. Fix, by our hypothesis, a family $\mathcal{A}%
=\{A_{n}: n\in \mathbb{N}\}$ of $k$-sized subsets of $X$ such that $A_i\neq
A_j$ for each pair of distinct natural numbers $i,j$. By disjointifying $%
\mathcal{A}$, if necessary, we may assume that $\mathcal{A}$ is 
disjoint and each of its members has size $\leq k$. Fix an $m$-element
subset $B=\{x_{1},x_{2},...,x_{m}\}$ of $X$. Since $\mathcal{A}$ is
disjointed, it follows that only finitely many members of $\mathcal{A}$ can
meet $B$. Assume that no member of $\mathcal{A}$ meets $B$. For every $n\in 
\mathbb{N}$, let $k_{n}=m-|A_{n}|$ and define $A_{n}^{\prime }=A_{n}\cup
\{x_{i}: i\leq k_{n}\}$. It is easy to see that $\mathcal{A}^{\prime
}=\{A_{n}^{\prime }: n\in \mathbb{N}\}$ is a denumerable family of $m$%
-element sets. Hence $\mathbf{IDI}_{m}$ is true.\smallskip

(ii) Fix an infinite set $X$ and let, by $\mathbf{CMC}$, $\{A_{n}: n\in 
\mathbb{N}\}$ be a family of non-empty finite sets such that $A_n\subseteq [X]^{n}$ for each $n\in\mathbb{N}$.
Since, for every $n\in \mathbb{N}$, $\bigcup A_{n}$ is a finite set of size $%
\geq n$, we can construct, via a straightforward induction, a subfamily $%
\{A_{k_{n}}: n\in \mathbb{N}\}$ of $\{A_{n}: n\in \mathbb{N}\}$ such that, for
all $n,m\in \mathbb{N}$, if $n<m$, then $k_{n}<k_{m}$ and $|\bigcup
A_{k_{n}}|<|\bigcup A_{k_{m}}|$. Clearly, $\{\bigcup A_{k_{n}}: n\in \mathbb{N%
}\}$ is a denumerable family of finite subsets of $X$. Hence, $X$ is quasi
Dedekind-infinite.\smallskip

(iii) This assertion is straightforward. \smallskip

(iv) It is obvious that $\mathbf{IDI}$ implies $\mathbf{IDI}_{2}$. For the
second assertion, let $\mathcal{M}$ be a $\mathbf{ZF}$ model including a
family $\mathcal{A}=\{A_{n}: n\in \mathbb{N}\}$ of two-element sets without a
partial choice, e.g., Cohen's Second Model $\mathcal{M}7$ in \cite{hr}.
Then, in $\mathcal{M}$, the set $X=\bigcup \mathcal{A}$ is Dedekind-finite,
but $\mathcal{A}$ is a countably infinite subset of $[X]^{2}$.

(v) We leave the proof of (v) as an easy exercise for the reader.\medskip
\end{proof}

\begin{proposition}
\label{p06.4} Let $\mathcal{N}$ be any model of $\mathbf{ZFA}$ satisfying $%
\mathbf{CMC}$ together with $(\forall n\in \omega ,n\geq 2)C(\omega ,n)$ and
the negation of $\mathbf{IDI}$. For instance, let $\mathcal{N}$ be Levy's
Model I denoted by $\mathcal{N}$6 in \cite{hr}. Then $\mathbf{IQDI}$ holds
in $\mathcal{N}$ but $\mathbf{IDI}_{F}$ fails, i.e., for every $k\in \mathbb{%
N},k\geq 2,\mathbf{IDI}_{k}$ fails in $\mathcal{N}$.
\end{proposition}

\begin{proof}
By part (ii) of Theorem \ref{t06.3}, $\mathbf{IQDI}$ holds in $\mathcal{N}$%
. Assume, aiming for a contradiction, that $\mathbf{IDI}_{k}$ holds in $%
\mathcal{N}$ for some $k\in 
\mathbb{N}$. Fix an infinite Dedekind-finite set $X\in \mathcal{N}$. By $%
\mathbf{IDI}_{k}$, there exists a denumerable family $\mathcal{A}%
=\{A_{n}: n\in \mathbb{N}\}$ of $k$-sized subsets of $X$. Let, by Lemma \ref{l06.2}, $r$ be a root of an infinite subfamily $%
\mathcal{B}$ of $\mathcal{A}$. Clearly, $\{B\backslash r: B\in \mathcal{B}\}$
is a denumerable family of pairwise disjoint subsets of $X$, each of size $%
k\backslash |r|$. Hence, by $C(\infty ,k\backslash |r|),$ $\{B\backslash
r: B\in \mathcal{B}\}$ has a choice set $C$. Since $C$ is clearly
denumerable, it follows that $X$ is Dedekind-infinite. Contradiction!
Therefore, in $\mathcal{N}$, $\mathbf{IQDI}$ holds but, for every $k\in 
\mathbb{N}$,  $\mathbf{IDI}_{k}$ does not hold.
\end{proof}

\begin{remark}
\label{r06.5} In \cite{et}, it is shown that $\mathbf{IDI}$ implies $\mathbf{%
I0}dim\mathbf{HS}(cell,\aleph _{0})$ and that this implication is not
reversible in $\mathbf{ZFA}$. We notice that, in the second Fraenkel model
(model $\mathcal{N}$2 in \cite{hr}), $\mathbf{CMC}$ holds but $\mathbf{IDI}$
fails. It follows from Theorem \ref{t06.3} that $\mathbf{IQDI}$ holds in $%
\mathcal{N}$2. Hence $\mathbf{IQDI}$ does not imply $\mathbf{IDI}$ in $%
\mathbf{ZFA}$. E. Tachtsis \cite{ltah}  has informed us recently that the
result of Proposition \ref{p06.4} transfers to $\mathbf{ZF}$. Hence, $%
\mathbf{IQDI}$  implies none of $\mathbf{IDI},$ $\mathbf{IDI}_{2}$ and $%
\mathbf{IDI}_{F}$ in $\mathbf{ZF}$. Therefore, in view of Theorem \ref{t04.7}%
, the implication $\mathbf{IDI}$ $\rightarrow $ $\mathbf{I0}dim\mathbf{HS}%
(cell,\aleph _{0})$ is not reversible in $\mathbf{ZF}$.
\end{remark}

\begin{question}
\label{q06.6} Does $\mathbf{IDI}_{2}$ imply $\mathbf{IDI}?$
\end{question}

In the sequel, we use the new selection principles $\mathbf{PKW}(\infty,
<\aleph_0)$ and $\mathbf{QPKW}(\infty, <\aleph_0)$, both defined in Section
1. We aim to prove that, in every model $\mathcal{M}$ of $\mathbf{ZF+QPKW%
}(\infty, <\aleph_0)$, the sentences $\mathbf{IQDI}$ and $\mathbf{IHS}%
(cell, \aleph_0)$ are equivalent and imply that every infinite Boolean
algebra has a tower. To do this, let us begin with the trivial observation
that the following condition is satisfied in $\mathbf{ZF}$:

\begin{description}
\item[(6)] If $\mathcal{B}_0$ is a Boolean subalgebra of a Boolean algebra $%
\mathcal{B}$, then every tower of $\mathcal{B}_0$ is a tower of $\mathcal{B}$%
. In particular, if a Boolean algebra $\mathcal{B}$ has a Boolean subalgebra 
$\mathcal{B}_0$ such that $\mathcal{B}_0$ has a tower, then $\mathcal{B}$
has a tower.
\end{description}

The following proposition shows that, in a model of $\mathbf{ZF}$, an
infinite Boolean algebra $\mathcal{B}$ can have a tower but an infinite
Boolean subalgebra of $\mathcal{B}$ may fail to have a tower.

\begin{proposition}
\label{p06.7} Let $\mathcal{M}$ be any model of $\mathbf{ZF+\neg IQDI}$ (for
instance, let $\mathcal{M}$ be the model mentioned in Remark \ref{r04.8}).
Then it holds in $\mathcal{M}$ that there exists an infinite Hausdorff space 
$\mathbf{X}$ such that the Boolean algebra $RO(\mathbf{X})$ has a tower but
some infinite Boolean subalgebra of $RO(\mathbf{X})$ does not have a tower.
\end{proposition}

\begin{proof} We work inside $\mathcal{M}$. Let $\mathbf{X}_1=(X_1,\mathcal{T}_1)$ be any non-discrete first-countable Hausdorff space in $\mathcal{M}$. By Propositions \ref{p2.10} and \ref{p02.13}, $RO(\mathbf{X}_1)$ has a tower. It follows from Theorem \ref{t04.7} and Proposition \ref{p2.10} that there exists in $\mathcal{M}$ an infinite Hausdorff space $\mathbf{X}_2=(X_2, \mathcal{T}_2)$ such that $RO(\mathbf{X}_2)$ does not have a tower. We may assume that $X_1\cap X_2=\emptyset$. Let $\mathbf{X}=\mathbf{X}_1\oplus\mathbf{X}_2$ be the direct sum of $\mathbf{X}_1$ and $\mathbf{X}_2$, and let $\mathcal{B}_0$ be the Boolean subalgebra of $RO(\mathbf{X})$ generated by $RO(\mathbf{X}_2)$. Then $RO(\mathbf{X})$ has a tower, while $\mathcal{B}_0$ does not have a tower.  
\end{proof}

\begin{theorem}
\label{t06.8} $\mathbf{(ZF)}$

\begin{enumerate}
\item[(i)] $\mathbf{IQDI}$ implies that every infinite Boolean algebra has a
tower iff every Boolean algebra expressible as a denumerable union of finite
sets has a tower.

\item[(ii)] The conjunction of $\mathbf{IQDI}$ and $\mathbf{QPKW}(\infty,
<\aleph_0)$ implies that every infinite Boolean algebra has a tower.

\item[(iii)] $\mathbf{QPKW}(\infty, <\aleph_0)$ implies, for every
topological space $\mathbf{X}$, if $\mathbf{X}$ has a regular matrix, then $%
\mathbf{X}$ admits a denumerable cellular family.
\end{enumerate}
\end{theorem}

\begin{proof}
To prove (i) and (ii), we assume $\mathbf{IQDI}$ and fix an infinite Boolean algebra $\mathcal{D}=(\mathcal{D}, +, \cdot, \mathbf{0}, \mathbf{1})$. By $\mathbf{IQDI}$, there exists a family $\mathcal{B}=\{\mathcal{B}_{n}: n\in \omega\}$ of pairwise distinct
finite subsets of $\mathcal{D}$.  For every $n\in \omega$, let $\mathcal{D}_{n}=\bigcup_{i\in n+1}\mathcal{B}_i$ and let $\mathcal{G}_{n}$ be the Boolean subalgebra of $\mathcal{D}$ generated by $\mathcal{D}_{n}$. Since $\mathcal{G}_n\subseteq \mathcal{G}_{n+1}$  and $\mathcal{G}_n$ is finite for each $n\in\omega$, while the set $\mathcal{G}=\bigcup_{n\in\omega}\mathcal{G}_n$ is infinite,  without loss of generality, we may assume that $\mathcal{G}_n$ is a proper subset of $\mathcal{G}_{n+1}$ for every $n\in\omega$. 

 To conclude the proof of (i), we notice that $\mathcal{G}$ is a Boolean subalgebra of $\mathcal{D}$ and $\mathcal{G}$ is expressible as a denumerable union of finite sets; furthermore, it follows from (6) that if $\mathcal{G}$ has a tower, then $\mathcal{D}$ has a tower.  

To prove (ii), we assume both $\mathbf{IQDI}$ and $\textbf{QPKW}(\infty, <\aleph_0)$. In view of Proposition \ref{p02.12}, to show that $\mathcal{G}$ has a tower, it suffices to prove that $\mathcal{G}$ is Dedekind-infinite. 

Let $n\in\omega$. Since the Boolean algebra $\mathcal{G}_n$ is finite, it is atomic. Let $C_n$ be the set of all atoms of $\mathcal{G}_n$. It is known from the theory of finite Boolean algebras that the following condition is satisfied:
\begin{enumerate}
\item[(a)] for  every non-zero element $x$ of $\mathcal{G}_n$, there exists  a unique non-empty set $C(x)\subseteq C_n$ such that $x$ is the sum $\sum C(x)$ of all elements of $C(x)$.
\end{enumerate}
Moreover, for every $n\in\omega$, the Boolean algebra $\mathcal{G}_n$ is isomorphic with the power set algebra $\mathcal{P}(C_n)$. Hence, for every $n\in\omega$, the set $E_n=C_{n}\setminus C_{n+1}$ is non-empty. For $n\in\omega$ and $x\in C_n$, let $A(n, x)$ be the unique subset of $C_{n+1}$ such that $x=\sum A(n, x)$. We notice that if $n\in\omega$ and $x\in E_n$, then $A(n, x)$ is a finite set which consists of at least two elements.  By $\mathbf{QPKW}(\infty,< \aleph_0)$, there exist an infinite subset $J$ of $\bigcup_{n\in\omega} (\{n\}\times E_n)$ and a family $\{B(n, x): (n, x)\in J\}$ of non-empty sets such that, for every $(n, x)\in J$, $B(n, x)$ is a proper subset of $A(n, x)$. Let 
$$N=\{n\in\omega: \text{ there exists } x\in E_n \text{ such that } (n, x)\in J\}.$$
Since each $E_n$ is finite and $J$ is infinite, it follows that $N$ is infinite. Now, for each $n\in N$, we define 
$$t_n=\sum\{ t: t\in\bigcup\{B(n, x): (n, x)\in J\}\}.$$
Clearly, $t_n\in\mathcal{G}_{n+1}$. Suppose that $n\in\omega$ is such that  $t_n\in\mathcal{G}_n$. There exists a unique set $C(t_n)\subseteq C_n$ such that $t_n=\sum\{x: x\in C(t_n)\}=\sum\{t: t\in\bigcup\{A(n, x): x\in C(t_n)\}\}$. Then $\bigcup\{A(n, x): x\in C(t_n)\}=\bigcup\{B(n, x): (n, x)\in J\}$. Since the last equality is impossible, we deduce that $t_n\in\mathcal{G}_{n+1}\setminus\mathcal{G}_n$ for every $n\in\omega$.  This proves that $\mathcal{G}$ is Dedekind-infinite. By Proposition \ref{p02.12}, $\mathcal{G}$ has a tower, so $\mathcal{D}$  has a tower by (6). Hence (ii) holds.

To prove (iii), we assume $\mathbf{QPKW}(\infty, <\aleph_0)$ and fix a topological space $\mathbf{X}$ such that $\mathbf{X}$ admits a regular matrix $\mathcal{F}=\{\mathcal{F}_n: n\in\mathbb{N}\}$. Now, for every $n\in\omega$, let $\mathcal{B}_n=\bigcup\limits_{i=1}^{n+1}\mathcal{F}_n$ and let $\mathcal{G}_n$ be the Boolean subalgebra of $RO(\mathbf{X})$ generated by $\mathcal{B}_n$. Mimicking the proof to (ii), we can show that the Boolean subalgebra $\mathcal{G}=\bigcup_{n\in\omega}\mathcal{G}_n$ of $RO(\mathbf{X})$ has a tower. Hence $RO(\mathbf{X})$ has a tower, so $\mathbf{X}$ has a denumerable cellular family by Proposition \ref{p02.12}. This completes the proof.
\end{proof}

The following corollary shows that Theorem \ref{t06.8} leads to a positive
answer to Problem \ref{q2} in every model of $\mathbf{ZF+QPKW}(\infty,
<\aleph_0)$:

\begin{corollary}
\label{c06.9} In every model of $\mathbf{ZF+QPKW}(\infty, <\aleph_0)$, the
following conditions are all equivalent:

\begin{enumerate}
\item[(i)] $\mathbf{IQDI}$;

\item[(ii)] $\mathbf{I0}dim\mathbf{HS}(cell,\aleph_0)$;

\item[(iii)] $\mathbf{IHS}(cell, \aleph_0)$;

\item[(iv)] every infinite Tychonoff space has a denumerable cellular family;

\item[(v)] for every infinite set $X$, every infinite subspace of the
Tychonoff cube $[0,1]^X$ admits a denumerable cellular family;

\item[(vi)] every infinite Hausdorff space has a regular matrix.
\end{enumerate}
\end{corollary}

\begin{proof} Let $\mathcal{M}$ be a model of $\mathbf{ZF+QPKW}(\infty, <\aleph_0)$. It follows from Theorem \ref{t04.7} that (i) and (ii) are equivalent in $\mathcal{M}$. Since it holds in $\mathbf{ZF}$ that, for every infinite Tychonoff space $\mathbf{Y}$, there exists an infinite set $X$ such that $\mathbf{Y}$  is homeomorphic with an infinite subspace of the Tychonoff cube $[0, 1]^X$, it follows that conditions (iv) and (v) are equivalent in $\mathbf{ZF}$. Of course, (iii) implies (iv). If (iii) holds, then, for every infinite set $X$, the Cantor cube $\mathbf{2}^X$ has a denumerable cellular family because $\mathbf{2}^X$ is an infinite Tychonoff space. Hence (iv) implies (i) by Theorem \ref{t04.7}. 
 
 Now, assume that $\mathbf{X}$ is an infinite Hausdorff space in $\mathcal{M}$. If (i) holds in $\mathcal{M}$, it follows from Theorem \ref{t06.8} that the Boolean algebra $RO(\mathbf{X})$ has a tower in $\mathcal{M}$, so,  by Proposition \ref{p02.12}, $\mathbf{X}$ has a denumerable cellular family in $\mathcal{M}$. Hence (i) implies (iii) in $\mathcal{M}$. In consequence, conditions (i)-(v) are all equivalent in $\mathcal{M}$.  Moreover, by Theorem \ref{t04.4}, (i) implies (vi) in $\mathcal{M}$. To complete the proof, we notice that, in view of Theorem \ref{t06.8},  (vi) implies (iii) in $\mathcal{M}$. 
 \end{proof}

Since $\mathbf{PKW}(\infty, <\aleph_0)$ and $\mathbf{QPKW}(\infty ,<\aleph
_{0})$ are new here, let us scrutinize a little bit on their set theoretic
strength. To do this, we also need the following forms:

\begin{itemize}
\item $\mathbf{PKW}(\infty, \leq n)$ where $n\in\omega\setminus\{0,1\}$: For
every infinite set $J$ and every family $\{A_{j}: j\in J\}$ of finite sets
such that $1<|A_{j}|\leq n$ for every $j\in J$, there exist an infinite
subset $I$ of $J$ and a family $\{B_{j}: j\in I\}$ of non-empty sets such
that, for every $j\in I$, $B_{j}$ is a proper subset of $A_{j}$.

\item $\mathbf{PAC}(\leq n)$ where $n\in\omega\setminus\{0,1\}$: Every
infinite family $\mathcal{A}$ of non-empty at most $n$-element sets has an
infinite subfamily $\mathcal{A}^{\prime}$ such that $\mathcal{A}^{\prime}$
has a choice function.

\item $\mathbf{UPKWF}$: For every $n\in\omega\setminus\{0,1\}$, $\mathbf{PKW}%
(\infty,\leq n)$.

\item $\mathbf{UPACF}$: For every $n\in\omega\setminus\{0, 1\}$, $\mathbf{PAC%
}(\leq n)$.

\item $\mathbf{PCAC}(\leq n)$ where $n\in\omega\setminus\{0,1\}$: Every
denumerable family of non-empty at most $n$-element sets has a partial
choice function.
\end{itemize}

Let us notice that $\mathbf{PAC}(\leq 2)$ is equivalent to Form 166 of \cite%
{hr}.

\begin{proposition}
\label{p06.10}$(\mathbf{ZF})$ For every $n\in\omega\setminus\{0,1\}$, the
following implications and equivalences hold:

\begin{enumerate}
\item[(i)] $\mathbf{CAC}_{fin}\rightarrow \mathbf{QPKW}(\infty ,<\aleph
_{0})\rightarrow \mathbf{PCAC}(\leq n)$;

\item[(ii)] $\mathbf{PKW}(\infty ,\leq n)\leftrightarrow \mathbf{PAC}(\leq n)
$ and $\mathbf{UPKWF}\leftrightarrow\mathbf{UPACF}$;

\item[(iii)] $\mathbf{PKW}(\infty, <\aleph_0)\leftrightarrow(\mathbf{QPKW}%
(\infty, <\aleph_0)\wedge\mathbf{UPKWF})$;

\item[(iv)] $\mathbf{PKW}(\infty, <\aleph_0)\leftrightarrow(\mathbf{QPKW}%
(\infty, <\aleph_0)\wedge\mathbf{UPACF})$.
\end{enumerate}
\end{proposition}

\begin{proof} (i) For the first implication, assume $\mathbf{CAC}_{fin}$ and fix a family $\mathcal{A}=\{A_j: j\in J\}$ of finite sets such that $\vert A_j\vert\ge 2$ for each $j\in J$. Assume that $J=\bigcup_{n\in\omega}J_n$ where each $J_n$ is a non-empty finite set, and $J_m\cap J_n=\emptyset$ for each pair $m, n$ of distinct members of $\omega$. By $\mathbf{CAC}_{fin}$, we can choose $f\in\prod_{n\in\omega}J_n$ and $g\in\prod_{n\in\omega}A_{f(n)}$. Let $I=\{f(n): n\in\omega\}$. The set $I$ is infinite and, for each $i\in I$, there is a unique $n(i)\in\omega$ such that $i=f(n(i))$. Then, for each $i\in I$, $B_i=\{g(f(n(i))\}$ is a non-empty proper subset of $A_{f(n(i))}$. Hence $\mathbf{CAC}_{fin}$ implies $\mathbf{QPKW}(\infty, <\aleph_0)$.
 
To prove that the second implication of (i) holds, fix a disjoint family $\mathcal{E}=\{E_{i}: i\in\omega\}$ of non-empty at most $n$%
-element sets. Let us assume $\mathbf{QPKW}(\infty,< \aleph_0)$ and suppose that $\mathcal{E}$ does not have a partial choice function.  Via a straightforward induction, for each $k\in n$, we find an infinite subset $N_k$ of $\omega$ and a family $\{D_{k,i}: i\in N_k\}$ of non-empty sets such that, for each $i\in N_k$, $D_{k,i}$ is a proper subset of $E_i$ and, moreover $N_{k+1}\subseteq N_k$ for each $k\in n-1$. To begin the induction, we use $\mathbf{QPKW}(\infty, <\aleph_0)$ to fix an infinite set $N_0\subseteq\omega$ and a family $\{D_{0,i}: i\in N_0\}$ such that, for each $i\in N_0$, $D_{0,i}$ is a non-empty proper subset of $E_i$. Suppose that $k\in n$ is such that we have already defined an infinite set $N_k\subset\omega$ and a family $\{D_{k,i}: i\in\omega\}$ of non-empty sets such that $D_{k,i}\subset E_i$ for each $i\in N_k$. Since $\mathcal{E}$ does not have a partial choice function, we may assume that $1<\vert D_{k,i}\vert$ for each $i\in N_k$. By $\mathbf{QPKW}(\infty,<\aleph_0)$, there exists an infinite set $N_{k+1}\subseteq N_k$ and a family $\{D_{k+1,i}: i\in N_{k+1}\}$ of non-empty sets such that $D_{k+1, i}\subset D_{k, i}$ for each $i\in N_{k+1}$. This terminates our induction. We notice that $\vert D_{k, i}\vert\leq n-i-1$ for each $k\in n$ and each $i\in N_k$. In particular $D_{n-1, i}=\emptyset$ for each $i\in N_{n-1}$. The contradiction obtained shows that $\mathcal{E}$ has a partial choice function if $\mathbf{QPKW}(\infty,< \aleph_0)$  holds. 

By mimicking and modifying a little the proof of (i), we can prove the first equivalence of (ii). The second equivalence of (ii) follows from the first one. It follows from (ii) and (iii) that (iv) holds. 

To prove (iii), we notice that, trivially, $\mathbf{PKW}(\infty ,<\aleph _{0})$ implies both $\mathbf{QPKW}%
(\infty ,<\aleph _{0})$ and $\mathbf{UPKWF}$. On the other hand, given a
family $\mathcal{A}=\{A_{j}: j\in J\}$ of finite sets such that the set $J$ is infinite and $|A_{j}|>1$
for every $j\in I$, we consider the following cases:

(a) $J$ is countable. In this case, $\mathbf{QPKW}(\infty ,<\aleph _{0})$ implies that the conclusion of $\mathbf{PKW}(\infty ,<\aleph _{0})$ holds for the family $\mathcal{A}$.

(b) $J$ is uncountable. For every $n\in \mathbb{N}$ let $J_{n}=\{j\in
J: |A_{j}|=n\}$.  We consider the following subcases:

(b1) For some $n\in \mathbb{N},$ $J_{n}$ is infinite. Then $\mathbf{UPKWF%
}$ implies that there exists an infinite subset $I$ of $J_{n}$
and a family $\{B_{j}: j\in I\}$ of non-empty sets such that, for every $j\in
I$, $B_{j}$ is a proper subset of $A_{j}$.

(b2) For every $n\in \mathbb{N},$ $J_{n}$ is finite. In this case the
conclusion of the statement $\mathbf{PKW}(\infty ,<\aleph _{0})$ for $\mathcal{A}$ follows
from $\mathbf{QPKW}(\infty ,<\aleph _{0})$.
\end{proof}

\begin{proposition}
\label{p06.11} It holds in $\mathbf{ZFA}$ that neither $\mathbf{PKW}(\infty,
<\aleph_0)$ implies $\mathbf{IQDI}$ nor $\mathbf{IQDI}$ implies $\mathbf{QPKW%
}(\infty,<\aleph_0)$, nor $\mathbf{QPKW}(\infty,<\aleph_0)$ implies $\mathbf{%
UPKWF}$, nor $\mathbf{QPKW}(\infty,<\aleph_0)$ implies $\mathbf{PKW}(\infty,
<\aleph_0)$.
\end{proposition}

\begin{proof} It was proved in \cite{CruzP} that $\mathbf{PKW}(\infty, \infty, \infty)$ (Form 379 in \cite{hr}) holds in the Basic Fraenkel Model $\mathcal{N}$1 of \cite{hr}. Hence $\mathbf{PKW}(\infty,<\aleph_0)$ holds in $\mathcal{N}$1. Since the set of all atoms of $\mathcal{N}$1 is amorphous in $\mathcal{N}$1, it follows from Corollary \ref{c04.13} that $\mathbf{IQDI}$ fails in $\mathcal{N}$1. Hence 
$\mathbf{PKW}(\infty, <\aleph_0)$ does not imply $\mathbf{IQDI}$ in $\mathbf{ZFA}$.

In the Second Fraenkel Model $\mathcal{N}$2 of \cite{hr}, $\mathbf{CMC}$ holds. Hence, by Theorem \ref{t06.3}, $\mathbf{IQDI}$ also holds in $\mathcal{N}$2. However, there exists in $\mathcal{N}$2 a denumerable family of two-element sets which does not have a partial choice function (see page 178 of \cite{hr}). Therefore, by Proposition \ref{p06.10}, $\mathbf{QPKW}(\infty, <\aleph_0)$ fails in $\mathcal{N}$2. We remark that $\mathbf{UPKWF}$ also fails in $\mathcal{N}$2.

In Hickman's Model I (model $\mathcal{N}$24 in \cite{hr}), $\mathbf{CAC}_{fin}$ holds, so, by Proposition \ref{p06.10}, $\mathbf{QPKW}(\infty, <\aleph_0)$ is true in $\mathcal{N}$24. It is known that $\mathbf{PAC}(\leq 2)$ fails in $\mathcal{N}$24 (see page 200 in \cite{hr}). It follows from Proposition \ref{p06.10} that both $\mathbf{UPKWF}$ and $\mathbf{PKW}(\infty, <\aleph_0)$ are false in $\mathcal{N}$24. 
\end{proof}

Let us recall the following definition which can be found, for instance, in 
\cite{tr}:

\begin{definition}
\label{d0612}  A set $X$ is called \textit{strictly amorphous} if it does
not admit infinite partitions into finite sets having at least two elements.
\end{definition}

It is easy to verify that the following proposition holds:

\begin{proposition}
\label{p0613} It is true in $\mathbf{ZF}$ that $\mathbf{UPKWF}$ implies that
there are no strictly amorphous sets.
\end{proposition}

\begin{remark}
\label{r0614} By Proposition \ref{p06.4}, $\mathbf{IQDI}$ holds in Levy's
Model I ($\mathcal{N}$6 in \cite{hr}). Since Form 342 of \cite{hr} holds in $%
\mathcal{N}$6, it follows from Proposition \ref{p06.10} that $\mathbf{UPKWF}$
holds in $\mathcal{N}$6. It is known that $\mathbf{IDI}$, $\mathbf{CAC}_{fin}
$ and $\mathbf{KW}(\aleph_0, <\aleph_0)$(Form 358 of \cite{hr}) are false in 
$\mathcal{N}$6 (see page 186 in \cite{hr}). Hence, in $\mathbf{ZFA}$, the
conjunction of $\mathbf{UPKWF}$ and $\mathbf{IQDI}$ implies neither $\mathbf{%
CAC}_{fin}$ nor $\mathbf{IDI}$, nor $\mathbf{KW}(\aleph_0, <\aleph_0)$. We
do not know if it is possible to find a model of $\mathbf{ZF}$ in which both 
$\mathbf{IQDI}$ and $\mathbf{PKW}(\infty, <\aleph_0)$ hold but $\mathbf{CAC}%
_{fin}$ fails. We do not know a model of $\mathbf{ZF}$ in which $\mathbf{%
UPKWF}$ holds but $\mathbf{QPKW}(\infty, <\aleph_0)$ fails.
\end{remark}

It was proved in \cite{kt} that $\mathbf{IDI}$ is equivalent to the
conjunction of $\mathbf{CAC}_{fin}$ and the sentence ``Every infinite
Boolean algebra has a tower''. As an immediate consequence of Theorem \ref%
{t06.8}, taken together with Corollary \ref{c06.9} and Proposition \ref%
{p06.10}, we can state the following final proposition:

\begin{proposition}
The conjunction of $\mathbf{QPKW}(\infty, <\aleph_0)$ and the sentence
``Every infinite Boolean algebra has a tower'' implies $\mathbf{IQDI}$ and
follows from $\mathbf{CAC}_{fin}+\mathbf{IQDI}$.
\end{proposition}

\end{document}